\documentclass[11pt]{amsart}
\usepackage[utf8]{inputenc}
\usepackage{amsfonts}
\usepackage{amsmath}
\usepackage{amsthm}
\usepackage{amssymb}
\usepackage{mathtools}
\usepackage{graphicx} 
\usepackage{subcaption}
\usepackage{microtype}
\usepackage[dvipsnames]{xcolor}
\usepackage{tikz}
\usepackage{nicefrac}
\usetikzlibrary{cd, shapes}
\usepackage[normalem]{ulem}

\newtheorem{thm}{Theorem}[section]

\newtheorem{prop}[thm]{Proposition}
\newtheorem{cor}[thm]{Corollary}
\newtheorem{conj}[thm]{Conjecture}
\newtheorem{lemma}[thm]{Lemma}
\newtheorem{rmk}[thm]{Remark}
\theoremstyle{definition}

\newtheorem{dfn}[thm]{Definition}
\numberwithin{equation}{section}

\usepackage{hyperref}
\hypersetup{
    colorlinks=true,
    linkcolor=blue,
    citecolor=magenta,
    filecolor=magenta,      
    urlcolor=magenta,
}
\usepackage{float}

\usepackage[colorinlistoftodos]{todonotes}

\let\originalleft\left
\let\originalright\right
\renewcommand{\left}{\mathopen{}\mathclose\bgroup\originalleft}
\renewcommand{\right}{\aftergroup\egroup\originalright}

\newcommand{\ord}{\mathsf{ord}}

\newcommand{\neck}{\mathsf{Neck}}
\newcommand{\mneck}{\mathsf{MNeck}}
\newcommand{\comp}{\mathsf{Comp}}
\newcommand{\ccomp}{\mathsf{CComp}}
\newcommand{\pcomp}{\mathsf{PComp}}
\newcommand{\mob}{\mathsf{M\ddot{o}b}}

\title{A combinatorial proof of a symmetry for a refinement of the Narayana numbers}

\author{Mikl\'os B\'ona}
\thanks{B\'ona is partially supported by Simons Foundation Collaboration Awards 421967 and 940024.}
\address{Department of Mathematics\\
         University of Florida\\
\url{https://people.clas.ufl.edu/bona/}}
\email{bona@ufl.edu}

\author{Stoyan Dimitrov}
\address{Department of Mathematics\\
         Rutgers University\\
\url{https://stoyandimitrov.net/}}
\email{EmailToStoyan@gmail.com}

\author{Gilbert Labelle}
\address{LACIM \\ D\'epartement de math\'ematiques\\
         Universit\'e du Qu\'ebec \`a Montr\'eal
\url{}}
\email{labelle.gilbert@uqam.ca}

\author{Yifei Li}
\address{Department of Mathematical Sciences \& Philosophy \\
         University of Illinois at Springfield
\url{}}
\email{yli236@uis.edu}

\author{Joseph Pappe}
\thanks{Pappe is partially supported by the National Science Foundation under Awards DMS-1760329 and DMS-2053350.}
\address{Department of Mathematics\\
         Colorado State University\\
\url{https://sites.google.com/view/josephpappe/home/}}
\email{joseph.pappe@colostate.edu}

\author{Andr\'es R. Vindas-Mel\'endez}
\thanks{Vindas-Mel\'endez was partially supported by the National Science Foundation under Award DMS-2102921.}
\address{Department of Mathematics, 
Harvey Mudd College\\
\url{https://math.hmc.edu/arvm/}}
\email{avindasmelendez@g.hmc.edu}

\author{Yan Zhuang}
\thanks{Zhuang is partially supported by an AMS-Simons Travel Grant, and by the National Science Foundation under Award DMS-2316181.}
\address{Department of Mathematics and Computer Science\\
         Davidson College\\
\url{https://yanzhuang.name/}}
\email{yazhuang@davidson.edu}

\global\long\def\Nar{\operatorname{Nar}}%

\begin{document}


\begin{abstract}
We establish a tantalizing symmetry of certain numbers refining the Narayana numbers. 
In terms of Dyck paths, this symmetry is interpreted in the following way: if $w_{n,k,m}$ is the number of Dyck paths of semilength $n$ with $k$ occurrences of $UD$ and $m$ occurrences of $UUD$, then $w_{2k+1,k,m}=w_{2k+1,k,k+1-m}$.
We give a combinatorial proof of this fact, relying on the cycle lemma, and showing that the numbers $w_{2k+1,k,m}$ are multiples of the Narayana numbers. We prove a more general fact establishing a relationship between the numbers $w_{n,k,m}$ and a family of generalized Narayana numbers due to Callan. 
A closed-form expression for the even more general numbers $w_{n,k_{1},k_{2},\ldots , k_{r}}$ counting the semilength-$n$ Dyck paths with $k_{1}$ $UD$-factors, $k_{2}$ $UUD$-factors, $\ldots$ , and $k_{r}$ $U^{r}D$-factors is also obtained, as well as a more general form of the discussed symmetry for these numbers in the case when all rise runs are of certain minimal length. 
Finally, we investigate properties of the polynomials $W_{n,k}(t)= \sum_{m=0}^k w_{n,k,m} t^m$, including real-rootedness, $\gamma$-positivity, and a symmetric decomposition.
\end{abstract}

\maketitle

\section{Introduction}

A \emph{lattice path} is a path in the discrete integer lattice $\mathbb{Z}^n$ consisting of a sequence of steps from a prescribed \emph{step set} and satisfying prescribed restrictions.
Among classical lattice paths, Dyck paths are perhaps the most well-studied. 
A \emph{Dyck path} of \emph{semilength} $n$ is a path in $\mathbb{Z}^2$ with step set $S= \{ (1,1), (1,-1) \}$ that starts at the origin $(0,0)$, ends at $(2n,0)$, and never traverses below the horizontal axis. 
The steps $(1,1)$ are called \emph{up steps} and can be represented by the letter $U$, while the steps $(1,-1)$ are \emph{down steps} and denoted $D$. 
It is well known that the number of Dyck paths of semilength $n$ is the Catalan number $C_n=\binom{2n}{n}/(n+1)$ which is known to enumerate more than 250 other families of combinatorial objects \cite{stanleyCat}.

The number of Dyck paths of semilength $n$ that contain exactly $k$ places where a $U$ is immediately followed by a $D$---called $UD$-\emph{factors} (or \emph{peaks})---is equal to the \emph{Narayana number} $N_{n,k}$ defined by
\[N_{n,k} = \frac{1}{n}\binom{n}{k} \binom{n}{k-1}, \]
for all $1 \leq k \leq n$ and $N_{0,0} = 1$. Knowing this, it is easy to see that $N_{n,k}$ also counts the number of Dyck paths of semilength $n$ with $k-1$ $DU$-factors (\emph{valleys}). The Narayana numbers are known to exhibit the symmetry $N_{n,k} = N_{n,n+1-k}$, which can be proved combinatorially via several involutions; see, for example, \cite{Deutsch, Kreweras1970, Kreweras1972, Lalanne1992}.

There is interest in counting Dyck paths (and other kinds of lattice paths) by the number of occurrences of longer factors, or even finding the \emph{joint distribution} of occurrences of multiple kinds of factors.
Two early works in this direction are \cite{Deutsch1999, sapounakis-tsoulas-tsikouras}. 
In \cite{wang}, Wang discusses a general technique that is useful for obtaining the relevant generating functions in many such cases; see also \cite{yam}.
Our work is concerned with the joint distribution of $UD$-factors and $UUD$-factors over Dyck paths. This does not seem to have been studied before, and is different from the instances discussed in the works cited above as the factor $UD$ is an ending segment of the factor $UUD$. 
We can also interpret these factors in terms of \emph{rise runs}---maximal consecutive subsequences of up steps---as the number of $UD$-factors is equal to the total number of rise runs, and the number of $UUD$-factors is equal to the number of rise runs of length at least 2.

Let $w_{n,k,m}$ be the number of Dyck paths of semilength $n$ with $k$ $UD$-factors and $m$ $UUD$-factors. 
Then we have the following symmetry:
\begin{thm}
\label{th:main}
  For all $1 \leq m \leq k$, we have 
  $$
  w_{2k+1,k,m} = w_{2k+1,k, k+1-m}.
  $$
\end{thm}
\noindent In other words, among all Dyck paths of semilength $2k+1$ with $k$ $UD$-factors, the number of those with $m$ $UUD$-factors is equal to the number of those with $k+1-m$ $UUD$-factors. 
To prove this symmetry, the first and third authors used generating function techniques to derive the following closed formula for the numbers $w_{n,k,m}$ (see Section~\ref{sec:app} in the Appendix), from which Theorem \ref{th:main} readily follows.

\begin{thm} \label{t-explicit} We have
\begin{equation*} 
w_{n,k,m} = \begin{cases}
    \displaystyle{\frac{1}{k}\binom{n}{k-1}\binom{n-k-1}{m-1}\binom{k}{m}}, & \text{ if } 0<m\leq k, \, \text{and } k+m\leq n, \\
    1, & \text{ if } m = 0 \text{ and } n = k,\\
    0, & \text{ otherwise.}
\end{cases}
\end{equation*}
\end{thm}

Nonetheless, Theorem \ref{th:main} cries for combinatorial explanation, but a combinatorial proof eluded the first and third authors for several years.
The remaining authors joined this project in Summer 2022 as part of the AMS Mathematics Research Community \emph{Trees in Many Contexts}, during which we found a combinatorial proof for Theorem \ref{th:main}. This proof is the main focus of the current paper. In fact, we give a combinatorial proof for the following identity relating the numbers $w_{2k+1,k,m}$ to Narayana numbers.
\begin{thm}
\label{th:relationNandW}
For all $k \geq 1$ and $m\geq 0$, we have
\begin{equation}
\label{eq:relNW}
w_{2k+1,k, m} = \binom{2k+1}{k-1} N_{k,m}.    
\end{equation}
\end{thm}
Observe that Theorem \ref{th:main} follows immediately from Theorem \ref{th:relationNandW} and the symmetry of the Narayana numbers. 
Our combinatorial proof of Theorem \ref{th:relationNandW} requires heavy use of the cycle lemma.
We show that analogous statements to Theorem~\ref{th:main} and Theorem~\ref{th:relationNandW} hold for the numbers $w_{2k-1,k,m}$. We also generalize Theorem \ref{th:relationNandW} to a relationship between the $w_{n,k,m}$ and a family of generalized Narayana numbers introduced by Callan ~\cite{callan} (see Theorem~\ref{thm:simplify}). 

Using the ideas in the combinatorial proofs of Theorem~\ref{th:main} and Theorem~\ref{t-explicit}, we establish the following more general statements for $w_{n, k_{1}, k_{2}, \ldots, k_{r}}$, defined to be the number of Dyck paths of semilength $n$ with $k_{1}$ $UD$-factors, $k_{2}$ $UUD$-factors,  \dots , and $k_{r}$ $U^{r}D$-factors. 

\begin{thm} \label{c-evenmoretantalizing}
For all $r\geq 2$ and $1\leq m \leq k$, taking $k_1 = k_2 = \ldots = k_{r-1} = k$, we have
\begin{equation*}
w_{rk+1, k, k, \ldots, k, m} = w_{rk+1, k, k, \ldots, k, k+1-m}
\end{equation*}
and
\begin{equation*}
w_{rk-1, k, k, \ldots, k, m} = w_{rk-1, k, k, \ldots, k, k-m}.
\end{equation*}
\end{thm}

\noindent Note that $r=2$ recovers the symmetry of the numbers $w_{2k\pm 1,k,m}$ (our Theorems~\ref{th:main} and \ref{th:relationNandW2}). Observe also that in terms of rise runs, Theorem~\ref{c-evenmoretantalizing} describes a symmetry in the case when all rise runs in the considered Dyck paths are of length at least $r-1$.

\begin{thm}\label{c-wr}
Let $r\geq 1$, $k_1 \geq k_2 \geq \cdots \geq k_r \geq 0$, and $n \geq k_1 + k_2 + \cdots + k_r$. 
Take $\hat{k} = k_1 + k_2 + \cdots + k_{r-1}$. Then
\begin{multline*}
w_{n,k_1,k_2, \ldots, k_r} = \\
\begin{cases}
    \displaystyle{\frac{1}{k_{1}}\binom{n}{k_{1}-1} \binom{n-\hat{k}-1}{k_r-1} \binom{k_{1}}{k_1-k_2, k_2-k_3, \ldots, k_{r-1}-k_{r}, k_r }}, & \text{ if } k_r>0, \\[10bp]
    \displaystyle{\frac{1}{k_{1}}\binom{n}{k_{1}-1}\binom{k_{1}}{k_1-k_2, k_2-k_3, \ldots, k_{r-1}-k_{r}, k_r }}, & \text{ if } k_{r} = 0 \text{ and } n = \hat{k}, \\[2bp]
    0, & \text{ otherwise.}
\end{cases}
\end{multline*}
\end{thm}

\noindent Theorem \ref{c-wr} specializes to the formula for Narayana numbers upon setting $r=1$ and to Theorem~\ref{t-explicit} for $r=2$.

We note that the numbers $w_{n,k,m}$ have been independently studied under different guises in the Ph.D.\ theses of Wang and Lemus-Vidales.
Wang \cite[Theorem 2]{wang} gave a formula for the joint distribution of $UU$-factors and $UUD$-factors over Dyck paths of semilength $n$. 
Wang's formula is equivalent to our formula in Theorem \ref{t-explicit}, as a Dyck path of semilength $n$ with $k$ occurrences of $UD$ has $n-k$ occurrences of $UU$. 
Moreover, Lemus-Vidales \cite[Theorem 3.1.3]{lemus} gave an analogous formula for counting Dyck paths by ``short peaks'' (i.e., $UD$-factors not preceded by a $U$), $UUU$-factors, and $UUD$-factors. 
It is shown in \cite[Lemma 3.1.1]{lemus} that the number of $UD$-factors of a Dyck path is the sum of its number of short peaks and $UUD$-factors and that the semilength is equal to the sum of its number of short peaks, $UUU$-factors, and twice its number of $UUD$-factors. 
Hence, Theorem \ref{t-explicit} can be recovered from Lemus-Vidales's formula.
Though Wang and Lemus-Vidales' results provide alternate routes to obtain some of our results, their formulas are obtained by deriving an appropriate functional equation and applying Lagrange inversion; neither of them had a combinatorial proof.
One more relevant and recent result is by Fu and Yang \cite{fu2025group}, who found that $w_{n,k,m}$ also counts plane trees with $n$ edges, $k$ internal nodes, and $m$ internal nodes with degree larger than one.

This paper is structured as follows.\footnote{An algebraic proof (using generating functions) for the explicit formula in Theorem \ref{t-explicit} can be found in the Appendix. } 
\begin{itemize}
    
    \item Section \ref{preliminaries} reviews background material needed for our combinatorial proofs, including Dyck paths and cyclic compositions.
    
   \item In Section \ref{combinatorial_proofs}, we provide combinatorial proofs of Theorem ~\ref{t-explicit} and Theorem ~\ref{th:relationNandW} as well as a variant of Theorem ~\ref{th:relationNandW} for the numbers $w_{2k-1,k,m}$. In addition, we show that these proofs can be generalized to prove Theorem ~\ref{c-evenmoretantalizing} and Theorem ~\ref{c-wr}.

   \item In Section \ref{s-generalized}, we establish a more general formula relating the numbers $w_{n,k,m}$ to the generalized Narayana numbers. 

\item We conclude in Section \ref{polynomials} with some investigation of the polynomials \[W_{n,k}(t)= \sum_{m=0}^k w_{n,k,m} t^m,\] including real-rootedness and $\gamma$-positivity results, as well as a symmetric decomposition.
Several conjectures on the polynomials $W_{n,k}(t)$ and their symmetric decomposition are given.
\end{itemize}



\section{Preliminaries}\label{preliminaries}

Here, we introduce the background material necessary for the forthcoming combinatorial proofs.

\subsection{Dyck paths}

Recall that a \emph{Dyck path} of \emph{semilength} $n$ is a path in $\mathbb{Z}^2$ that begins at the origin, ends at $(2n,0)$, never goes below the horizontal axis, and consists of a sequence of \emph{up steps} $(1,1)$ and \emph{down steps} $(1,-1)$. 
We can represent Dyck paths as \emph{Dyck words}: words $\pi$ on the alphabet $\{U,D\}$ with the same number of $U$'s and $D$'s, such that there are never more $D$'s than $U$'s in any prefix of $\pi$. 
When we refer to a $UD$- or $UUD$-factor in a Dyck path, we really mean a factor in the corresponding Dyck word.


\subsection{Cyclic compositions and the cycle lemma}

Given a sequence $p = p_1 p_2 \cdots p_n$, we say that a sequence $p^\prime$ is a \emph{cyclic shift} (or \emph{cyclic rotation}) of $p$ if $p^\prime$ is of the form
\[p^\prime = p_i p_{i+1} \cdots p_n p_1 p_2 \cdots p_{i-1},\]
for some $1 \leq i \leq n$.
Let us write $p \sim p^\prime$ whenever $p$ and $p^\prime$ are cyclic shifts of each other.

Let $\comp_{n,k}$ denote the set of all compositions of $n$ into $k$ parts, i.e., a sequence of $k$ positive integers whose sum is $n$.
We define a \emph{cyclic composition} $[\mu]$ to be the equivalence class of a composition $\mu$ under cyclic shift.
Let $\ccomp_{n,k}$ be the set of cyclic compositions consisting of compositions of $n$ into $k$ parts, which is well-defined because the number of parts of a composition and the sum of its parts are clearly invariant under cyclic shift. We define the \emph{order} of a cyclic composition $[\mu]$, denoted by $\ord[\mu]$, to be the number of representatives of $[\mu]$---that is, the number of distinct compositions that can be obtained from cyclically shifting $\mu$.
Note that applying $\ord[\mu]$ cyclic shifts to $\mu$ will return back $\mu$.
If $[\mu] \in \ccomp_{n,k}$ has order $k$, then we say that $[\mu]$ is \emph{primitive}. 

For any $[\mu] \in \ccomp_{n,k}$, there exists a positive integer $d$ dividing both $n$ and $k$ such that $[\mu]$ is a \emph{concatenation} of $d$ copies of a primitive cyclic composition $[\nu] \in \ccomp_{n/d,k/d}$, which means that there exists $\bar{\nu}\in[\nu]$ for which $\mu$ is a concatenation of $d$ copies of $\bar{\nu}$.
In this case, $\ord[\mu]=k/d=\ord[\nu]$. 
(If $d=1$, then $[\mu]$ itself is primitive and is a concatenation of itself.) 
For example, the cyclic composition $[1,2,1,1,2,1]$ is a concatenation of two copies of the primitive cyclic composition $[1,2,1]$, and both of these cyclic compositions have order 3.
Observe that this decomposition of cyclic compositions into primitive cyclic compositions is unique.

\begin{lemma}\label{lemma:2k+1_primitive}
If $n$ and $k$ are relatively prime, then $[\mu] \in \ccomp_{n,k}$ is primitive.
\end{lemma}
\begin{proof}
Let $[\mu] \in \ccomp_{n,k}$. 
Then $[\mu]$ can be uniquely decomposed as a concatenation of $d$ copies of a primitive cyclic composition, where $d$ is a common divisor of $n$ and $k$.
Since $n$ and $k$ are relatively prime, it follows that $d=1$, whence it follows that $[\mu]$ itself is primitive.
\end{proof}

The cycle lemma will play an important role in our proofs.
Given a positive integer $k$ and a sequence $p = p_{1}p_{2}\cdots p_{l}$ consisting only of $U$'s and $D$'s, we say that $p$ is \emph{$k$-dominating} if every prefix of $p$ (i.e.,  every sequence $p_{1}p_{2}\cdots p_{i}$ where $1 \leq i \leq l$) has more copies of $U$ than $k$ times the number of copies of $D$.

\begin{lemma}[Cycle lemma \cite{dvoretzky-motzkin}]
\label{lemma:cycle}
Let $k$ be a positive integer.
For any sequence $p = p_{1}p_{2}\cdots p_{m+n}$ consisting of $m$ copies of $U$ and $n$ copies of $D$, there are exactly $\text{max}(0,m-kn)$ cyclic shifts of $p$ that are $k$-dominating.
\end{lemma}

We refer to \cite{dershowit-zaks, dvoretzky-motzkin} for a proof of the cycle lemma as well as some applications.
We note that Raney \cite{raney} showed that the cycle lemma is equivalent to the Lagrange inversion formula; Raney's proof was later generalized to the multivariate case by Bacher and Schaeffer \cite{bacher-schaeffer}.

\begin{cor}[of the cycle lemma]
\label{cor:reverse_cycle}
Any sequence of $k$ copies of $\bigcirc$ and $k+1$ copies of $\square$ has exactly one cyclic shift with no proper prefix having more $\square$s than $\bigcirc$s.
\end{cor}
\begin{proof}

Given any sequence $\lambda$ of $k$ copies of $\bigcirc$ and $k+1$ copies of $\square$, let $\tilde{\lambda}$ be the reverse sequence of $\lambda$---that is, the sequence consisting of the entries of $\lambda$ but in reverse order.
The cycle lemma guarantees that there is exactly one cyclic shift of $\tilde{\lambda}$ that is $1$-dominating. 
The reverse sequence of this $1$-dominating cyclic shift is the cyclic shift of $\lambda$ that has no proper prefix having more $\square$s than $\bigcirc$s.
\end{proof}


\section{Combinatorial proof of formulas for \texorpdfstring{$w_{n,k,m}$}{w n,k,m}} \label{combinatorial_proofs}

\subsection{Combinatorial proof of Theorem ~\ref{t-explicit}}
We will give a combinatorial proof for the explicit formula of the numbers $w_{n,k,m}$ stated in Theorem \ref{t-explicit}.
This proof will require the notion of ``extended peaks" and the decomposition of a Dyck word into ``extended peaks".

\begin{dfn}
An \emph{extended peak} is a word in the alphabet $\{U,D\}$ consisting of a nonempty sequence of $U$'s followed by one $D$. 
Given an extended peak $S = U^aD$, define its \emph{up-length}, denoted by $\ell(S)$, as $\ell(S)\coloneqq a$.
\end{dfn}

We define a \emph{necklace of extended peaks} (or simply a \emph{necklace}) to be the equivalence class of a sequence of extended peaks under cyclic shift. 
Often it is more convenient for us to view a necklace as simply a collection of extended peaks with a given cyclic order; it will be clear from context when we do so.
Let $\neck_{n,k}$ denote the set of all necklaces with $k$ extended peaks and a total of $n$ $U$'s. 

Let $\psi$ be the map taking a composition $(\mu_1,\mu_2,\dots,\mu_k)$ of $n$ to the sequence $S_1 S_2\cdots S_k$ of extended peaks where $\ell(S_i)=\mu_i$ for each $i$. 
Note that $\psi$ is a bijection between compositions of $n$ with $k$ parts and sequences of $k$ extended peaks with a total of $n$ $U$'s; moreover, $\psi$ induces a bijection---which we also denote $\psi$ by a slight abuse of notation---from $\ccomp_{n,k}$ to $\neck_{n,k}$. 
To be precise, the necklace $\psi[\mu]$ is the equivalence class of $\psi(\bar{\mu})$ for any $\bar{\mu} \in [\mu]$, which clearly does not depend on the choice of representative. 

We define a \emph{marking} of a necklace of extended peaks $[S_{1}, \ldots, S_{k}]$ to be the necklace $[S_{1}, \ldots, S_{k}]$ with $k-1$ $U$'s marked. 
Given $[\mu] \in \ccomp_{n,k}$, let $\mneck[\mu]$ be the set of all marked necklaces of extended peaks corresponding to the cyclic composition $[\mu]$.
If $[\mu]$ is primitive---that is, if $\ord[\mu]=k$---then observe that the necklace $\psi[\mu]$ has $\binom{n}{k-1}$ distinct markings. 
More generally, we have the following:

\begin{lemma} \label{lemma:markednecksize}
Given $[\mu] \in \ccomp_{n,k}$, we have
\begin{equation*}
    \lvert \mneck[\mu] \rvert = \frac{\ord[\mu]}{k}\binom{n}{k-1}.
\end{equation*}
\end{lemma}

\begin{proof}
Let $N$ denote the necklace of extended peaks corresponding to $[\mu]$, and fix a sequence $S_1 S_2 \cdots S_k \in N$ of extended peaks. 
Then there are $\binom{n}{k-1}$ ways to choose the $k-1$ $U$'s to be marked in $S_1 S_2 \cdots S_k$. 
Upon taking all $k$ cyclic shifts of $S_1 S_2 \cdots S_k$, observe that each cyclic shift appears $k/\ord[\mu]$ times; accordingly, each of the markings counted by $\binom{n}{k-1}$ is $k/\ord[\mu]$ times the number of markings in $\mneck[\mu]$.
In other words, we have 
\begin{equation*}
    \frac{k}{\ord[\mu]} \lvert \mneck[\mu] \rvert = \binom{n}{k-1},
\end{equation*}
which is equivalent to our desired conclusion.
\end{proof}

A Dyck word $\pi = \pi_{1} \cdots \pi_{2n}$ with exactly $k$ $UD$-factors can be expressed uniquely in the form $\pi = U^{a_{1}}D^{b_{1}} \cdots U^{a_{k}}D^{b_{k}}$, where $(a_{1}, \ldots, a_{k})$ and $(b_{1}, \ldots, b_{k})$ are both compositions of $n$.
Let us call $(a_{1}, \ldots, a_{k})$ the \emph{rise composition} of $\pi$.
Given a cyclic composition $[\mu]$, denote by $\mathsf{D}[\mu]$ the set of all Dyck words with rise composition in the equivalence class $[\mu]$.

\begin{lemma} \label{lemma:risecompsize}
Given $[\mu] \in \ccomp_{n,k}$, we have
\begin{equation}
    \lvert \mathsf{D}[\mu] \rvert = \frac{\ord[\mu]}{k}\binom{n}{k-1}.
\end{equation}
\end{lemma}

\begin{proof}
From Lemma ~\ref{lemma:markednecksize}, it suffices to find a bijection from $\mneck[\mu]$ to $D[\mu]$. 

Let $[S] = [S_1, S_2, \ldots, S_k] \in \mneck[\mu]$. We will choose a unique ordering of the marked extended peaks $S_1, S_2, \ldots, S_k$. To do this, we first record the $k-1$ marked $U$'s and $k$ extended peaks using a sequence of $\bigcirc$'s and $\square$'s as follows. Starting with any extended peak $S_i$, record a $\bigcirc$ for each marked $U$ on this extended peak, and then record a $\square$ for this extended peak. Repeat this procedure for the next extended peak in the cyclic order until all extended peaks and their markings have been recorded.

This gives a sequence consisting of $k-1$ copies of $\bigcirc$ and $k$ copies of $\square$. It then follows from Corollary ~\ref{cor:reverse_cycle} that there is exactly one cyclic shift $\sigma = \sigma_1 \sigma_2 \ldots \sigma_{2k-1}$ of this sequence where every proper prefix of $\sigma$ has at least as many $\bigcirc$'s as the number of $\square$'s. Note that $\sigma_1 = \bigcirc$ and $\sigma_{2k-2}\sigma_{2k-1} = \square\square$. Using $\sigma$, we obtain a unique ordering $R_1, R_2, \ldots, R_k$ of the marked extended peaks of $[S]$ by taking $R_i$ to be the extended peak corresponding to the $i$th $\bigcirc$ in $\sigma$, for every $i\in [k]$. As this choice of ordering does not depend on the representative of $[S]$ chosen, the ordering $R_1, R_2, \ldots, R_k$ of the marked extended peaks in $[S]$ is well-defined. 

For example, consider the representative $(\color{red}\underline{U}\color{black}U\color{red}\underline{U}\color{black}D,UD,UUUD,\color{red}\underline{U}\color{black}D)$ in the marked necklace $[\color{red}\underline{U}\color{black}U\color{red}\underline{U}\color{black}D,UD,UUUD,\color{red}\underline{U}\color{black}D] \in \mneck[3,1,3,1]$. This is associated to the sequence $\bigcirc \bigcirc \square \square \square \bigcirc \square$ with only the cyclic shift $\bigcirc \square\bigcirc \bigcirc \square \square \square$ having at least as many $\bigcirc$'s as $\square$'s in every proper prefix.  Thus, $R_1 = \color{red}\underline{U}\color{black}D$, $R_2=\color{red}\underline{U}\color{black}U\color{red}\underline{U}\color{black}D$, $R_3 = UD$, and $R_4 = UUUD$.

Given the sequence of extended peaks $R_1, R_2, \ldots, R_k$, we construct a Dyck word $w$ in the following manner.
\begin{enumerate}
    \item Set $w = U^{\ell(R_1)}$ and mark the corresponding $U$'s that are marked in $R_1$.
    \item For the next $R_i$ that has not been appended to $w$, find the rightmost marked $U$ in $w$, and let its position in $w$ be $j$.
    Let $c$ be the number of $U$'s from $w_j$ to the end of $w$ inclusive that are not paired with a $D$ to their right. 
    Append $c$ copies of $D$ to the end of $w$ followed by $U^{\ell(R_i)}$.
    \item Mark the corresponding $U$'s in $w$ that were marked in $R_i$ and unmark $w_j$ in $w$.
    \item Repeat steps $(2)$ and $(3)$ until we have appended all $k$ extended peaks.
    \item Append enough $D$'s to the end of $w$ such that $w$ is a word of length $2n$.
\end{enumerate}
For example, given $R_1 = \color{red}\underline{U}\color{black}D$, $R_2=\color{red}\underline{U}\color{black}U\color{red}\underline{U}\color{black}D$, $R_3 = UD$, and $R_4 = UUUD$, the associated word $w$ is $UDUUUDUDDDUUUDDD$.

Note that there will always be at least one marked $U$ to indicate how many $D$'s must be added before the next extended peak because the number of marked $U$'s (the $\bigcirc$'s) will always be at least the number of extended peaks (the $\square$'s) that need attaching by the choice of ordering. Moreover, $w$ is a Dyck word as at every point in the construction of $w$, the number of $D$'s added are weakly less than the number of $U$'s present in $w$. The rise composition of $w$ is given by $(\ell(R_1), \ell(R_2), \ldots, \ell(R_k))$ which by construction is cyclically equivalent to $\mu$. Thus, $w$ is in $\mathsf{D}[\mu]$.

Conversely, consider a Dyck word $\pi = U^{a_1}D^{b_1} \cdots U^{a_k}D^{b_k}$ whose rise composition $(a_1,a_2,\ldots,a_k)$ is contained in the equivalence class $[\mu]$. We decompose $\pi$ into a marked sequence of extended peaks $(R_1, R_2, \ldots, R_k)$ as follows.
\begin{enumerate}
    \item Set $R_k = U^{a_k}D$ with no markings and set $w$ to be the word $\pi$ with $U^{a_k}D^{b_k}$ deleted from its end.
    \item For the largest $1 \leq i < k$ such that $R_i$ is not defined, delete $D^{b_i}$ from the end of $w$. Mark the $b_i$th rightmost $U$ of $w$ that is not paired with a $D$ in $w$. 
    \item Set $R_i$ to be the last $a_i$ $U$'s of $w$ including any markings. Append a $D$ to $R_i$ and delete $U^{a_i}$ from the end of $w$.
    \item Repeat steps (2) and (3) until $R_1, R_2, \ldots, R_k$ have all been defined.
\end{enumerate}
By construction, we have $\ell(R_i) = a_i$ for all $1 \leq i \leq k$.
Moreover, as $\pi$ is a Dyck word, the $U$ that is marked at the end of step $(2)$ must have been unmarked at the beginning of step $(2)$. Hence, $R_1, R_2, \ldots, R_k$ have exactly $k-1$ marked $U$'s, and $[R_1, R_2, \ldots, R_k]$ is an element of $\mneck[\mu]$.

It is straightforward to verify that the two procedures described above are inverse bijections between marked necklaces of extended peaks whose up-lengths are cyclically equivalent to $\mu$ and Dyck words whose rise compositions are cyclically equivalent to $\mu$. Hence, the lemma follows.
\end{proof}

Let $\comp_{n,k,m}$ denote the set of all compositions in $\comp_{n,k}$ with exactly $m$ parts at least $2$, and let $\ccomp_{n,k,m}$ be its cyclic counterpart. 
Using Lemma ~\ref{lemma:risecompsize}, we obtain a combinatorial proof for Theorem ~\ref{t-explicit}.  The proof of the nontrivial case is given below.

\begin{proof}[Combinatorial proof of Theorem ~\ref{t-explicit}]
First, we take the Dyck paths counted by $w_{n,k,m}$ and partition them by the cyclic equivalence classes of their rise compositions.
Then we have 
\begin{equation} \label{e-needthislater}
    w_{n,k,m} = \sum_{[\mu] \in \ccomp_{n,k,m}} \lvert D[\mu] \rvert
    = \sum_{[\mu] \in \ccomp_{n,k,m}} \frac{\ord[\mu]}{k} \binom{n}{k-1}
\end{equation}
upon applying Lemma  ~\ref{lemma:risecompsize}.
Next, recall that every cyclic composition $[\mu]\in\ccomp_{n,k,m}$ contains $\ord[\mu]$ distinct compositions in $\comp_{n,k,m}$, so we have 
\begin{equation} \label{e-wcomp}
    w_{n,k,m} = \sum_{\mu \in \comp_{n,k,m}} \frac{1}{\ord[\mu]}\frac{\ord[\mu]}{k} \binom{n}{k-1}
    = \frac{1}{k} \binom{n}{k-1} \lvert \comp_{n,k,m} \rvert.
\end{equation}
Finally, we claim that
\begin{equation} \label{e-compbin}
\lvert \comp_{n,k,m} \rvert = \binom{n-k-1}{m-1}\binom{k}{m};
\end{equation}
indeed, we can uniquely generate all compositions of $n$ into $k$ parts with exactly $m$ parts at least 2 using the following process:
\begin{enumerate}
    \item Take the composition $(1^k)$ consisting of $k$ copies of 1, and choose $m$ positions $1\leq i_1<i_2<\cdots<i_m \leq k$ within this composition; there are $\binom{k}{m}$ ways to do this.
    \item Choose a composition $\mu=(\mu_1,\mu_2,\dots,\mu_m)$ of $n-k$ into $m$ parts; there are $\binom{n-k-1}{m-1}$ ways to do this.
    \item For each $1 \leq j \leq m$, add $\mu_j$ to the $i_j$th entry of $(1^k)$. The result is a composition of $n$ into $k$ parts with exactly $m$ parts at least 2.
\end{enumerate}
Substituting \eqref{e-compbin} into \eqref{e-wcomp} completes the proof.
\end{proof}

As a consequence of our combinatorial proof for Theorem ~\ref{t-explicit}, we also obtain a proof of Theorem ~\ref{c-wr} for the numbers $w_{n, k_{1}, k_{2}, \ldots, k_{r}}$.


\subsection{Combinatorial proof of Theorem \ref{th:relationNandW}}

Now, we focus our attention on finding a combinatorial proof of Theorem \ref{th:relationNandW}, which in turn will give a combinatorial proof of Theorem~\ref{th:main}. See Figure ~\ref{im:symmetry} for an example of the symmetry in Theorem~\ref{th:main} that we wish to prove.

\begin{figure}[ht]
\begin{subfigure}[c]{\linewidth}
    \centering
    \begin{tikzpicture}[scale = .3, auto=center, rednode/.style={circle,fill=red!90, inner sep=1.5pt, minimum width=4pt}]
	\draw[step=1.0, gray!100, thin] (0,0) grid (10, 4);
	\draw[black!80, line width=2pt] (0,0) -- (1,1) -- (2,0) -- (3,1) -- (4,2) -- (5,3) -- (6,4) -- (7,3) -- (8,2) -- (9,1) -- (10,0) ;
    \draw (6,4) node[rednode] {};
    \begin{scope}[shift={(11,0)}]
        \draw[step=1.0, gray!100, thin] (0,0) grid (10, 4);
	\draw[black!80, line width=2pt] (0,0) -- (1,1) -- (2,2) -- (3,3) -- (4,4) -- (5,3) -- (6,2) -- (7,1) -- (8,0) -- (9,1) -- (10,0) ;
        \draw (4,4) node[rednode] {};
    \end{scope}
    \begin{scope}[shift={(22,0)}]
        \draw[step=1.0, gray!100, thin] (0,0) grid (10, 4);
	\draw[black!80, line width=2pt] (0,0) -- (1,1) -- (2,2) -- (3,3) -- (4,4) -- (5,3) -- (6,2) -- (7,1) -- (8,2) -- (9,1) -- (10,0) ;
        \draw (4,4) node[rednode] {};
    \end{scope}
    \begin{scope}[shift={(33,0)}]
        \draw[step=1.0, gray!100, thin] (0,0) grid (10, 4);
	\draw[black!80, line width=2pt] (0,0) -- (1,1) -- (2,2) -- (3,3) -- (4,4) -- (5,3) -- (6,2) -- (7,3) -- (8,2) -- (9,1) -- (10,0) ;
        \draw (4,4) node[rednode] {};
    \end{scope}
    \begin{scope}[shift={(44,0)}]
        \draw[step=1.0, gray!100, thin] (0,0) grid (10, 4);
	\draw[black!80, line width=2pt] (0,0) -- (1,1) -- (2,2) -- (3,3) -- (4,4) -- (5,3) -- (6,4) -- (7,3) -- (8,2) -- (9,1) -- (10,0) ;
        \draw (4,4) node[rednode] {};
    \end{scope}
    \end{tikzpicture}
    \caption{$w_{5,2,1} = 5$}
\end{subfigure}
\begin{subfigure}[b]{\linewidth}
    \centering
    \begin{tikzpicture}[scale = .3, auto=center, rednode/.style={circle,fill=red!90, inner sep=1.5pt, minimum width=5pt}]
	\draw[step=1.0, gray!100, thin] (0,0) grid (10, 4);
	\draw[black!80, line width=2pt] (0,0) -- (1,1) -- (2,2) -- (3,1) -- (4,0) -- (5,1) -- (6,2) -- (7,3) -- (8,2) -- (9,1) -- (10,0) ;
    \draw (2,2) node[rednode] {};
    \draw (7,3) node[rednode] {};
    \begin{scope}[shift={(11,0)}]
        \draw[step=1.0, gray!100, thin] (0,0) grid (10, 4);
	\draw[black!80, line width=2pt] (0,0) -- (1,1) -- (2,2) -- (3,3) -- (4,2) -- (5,1) -- (6,0) -- (7,1) -- (8,2) -- (9,1) -- (10,0) ;
        \draw (3,3) node[rednode] {};
        \draw (8,2) node[rednode] {};
    \end{scope}
    \begin{scope}[shift={(22,0)}]
        \draw[step=1.0, gray!100, thin] (0,0) grid (10, 4);
	\draw[black!80, line width=2pt] (0,0) -- (1,1) -- (2,2) -- (3,3) -- (4,2) -- (5,1) -- (6,2) -- (7,3) -- (8,2) -- (9,1) -- (10,0) ;
        \draw (3,3) node[rednode] {};
        \draw (7,3) node[rednode] {};
    \end{scope}
    \begin{scope}[shift={(33,0)}]
        \draw[step=1.0, gray!100, thin] (0,0) grid (10, 4);
	\draw[black!80, line width=2pt] (0,0) -- (1,1) -- (2,2) -- (3,1) -- (4,2) -- (5,3) -- (6,4) -- (7,3) -- (8,2) -- (9,1) -- (10,0) ;
        \draw (2,2) node[rednode] {};
        \draw (6,4) node[rednode] {};
    \end{scope}
    \begin{scope}[shift={(44,0)}]
        \draw[step=1.0, gray!100, thin] (0,0) grid (10, 4);
	\draw[black!80, line width=2pt] (0,0) -- (1,1) -- (2,2) -- (3,3) -- (4,2) -- (5,3) -- (6,4) -- (7,3) -- (8,2) -- (9,1) -- (10,0) ;
        \draw (3,3) node[rednode] {};
        \draw (6,4) node[rednode] {};
    \end{scope}
    \end{tikzpicture}
    \caption{$w_{5,2,2} = 5$}
\end{subfigure}
\caption{Dyck paths of semilength $5$ with $2$ peaks where peaks that are part of a $UUD$-factor are colored red.}
\label{im:symmetry}
\end{figure}

Our proof will mostly rely on two key results. The first gives a combinatorial interpretation for Narayana numbers in terms of cyclic compositions.

\begin{lemma}
\label{lemma:part1_path}
Let $k \geq 1$ and $m \geq 0$.
Then the Narayana number $N_{k,m}$ is the number of cyclic compositions of $2k+1$ into $k$ parts such that exactly $m$ parts are at least 2.
\end{lemma}

\begin{proof}
We will give a bijective map that takes a cyclic composition of $2k+1$ into $k$ parts, exactly $m$ of which are at least $2$, to a Dyck path of semilength $k$ with $m$ $UD$-factors, which are counted by the Narayana numbers $N_{k,m}$. 

Given a cyclic composition $[\mu_{1},\mu_{2},\ldots, \mu_{k}]$ of $2k+1$ with exactly $m$ parts that are at least $2$, consider the word $U^{\mu_1-1}D U^{\mu_2-1}D \cdots U^{\mu_k-1}D$, which has $k+1$ copies of $U$ and $k$ copies of $D$. 
By the cycle lemma, there is exactly one cyclic shift of this word that is $1$-dominating---that is, with more $U$'s than $D$'s in every prefix.
Then the first two entries of this $1$-dominating sequence are necessarily $U$'s. 
Removing the first $U$, we obtain a Dyck path of semilength $k$ with exactly $m$ $UD$-factors.

It is easily verified that the inverse procedure is given by the following: from a semilength $k$ Dyck path with $m$ $UD$-factors, we get a sequence $(a_1, a_2, \dots, a_k)$ where $a_i$ is the number of $U$'s that immediately precede the $i$th $D$.
For example, from $UDUUDDUD$ we get the sequence $(1,2,0,1)$. 
Then we add $2$ to $a_1$ and $1$ to each other $a_i$, forming a composition $\mu$ of $2k+1$ into $k$ parts, exactly $m$ of which are at least $2$.
Taking the cyclic composition $[\mu]$ completes the inverse.
\end{proof}

We note that the map used in the proof of Lemma \ref{lemma:part1_path} is related to the standard bijection between \L ukasiewicz paths and Dyck paths. 
A \emph{\L ukasiewicz path} of length $n$ is a path in $\mathbb{Z}^2$ with step set $\{(1,-1),(1,0),(1,1),(1,2), \ldots \}$, starting from $(0,0)$ and ending at $(n,0)$, that never traverses below the $x$-axis; these paths were introduced in relation to the preorder degree sequence of a plane tree, which determines the tree unambiguously \cite[Section 1.5.3]{flajolet-sedgewick}.

As $2k+1$ and $k$ are relatively prime, every cyclic composition of $\ccomp_{2k+1,k}$ is primitive. Lemma ~\ref{lemma:risecompsize} then gives us the following corollary.

\begin{cor}\label{cor:part2_trees}
Let $k\geq 1$. Given a cyclic composition $[\mu] \in \ccomp_{2k+1,k}$, there are exactly $\binom{2k+1}{k-1}$ Dyck words whose rise composition belongs to $[\mu]$.
\end{cor}

We are now ready to complete our combinatorial proof of Theorem \ref{th:relationNandW}.

\begin{proof}[Proof of Theorem \ref{th:relationNandW}]
Recall that $w_{2k+1, k, m}$ counts Dyck words of semilength $2k+1$, $k$ $UD$-factors, and $m$ $UUD$-factors; these are precisely the Dyck words of semilength $2k+1$ whose rise composition has $k$ parts with exactly $m$ parts having size at least 2. 
These rise compositions can be grouped into cyclic compositions of $2k+1$ having $k$ total parts and $m$ parts at least 2, which are counted by $N_{k,m}$ as established in Lemma \ref{lemma:part1_path}.
Furthermore, by Lemma \ref{cor:part2_trees}, there are exactly $\binom{2k+1}{k-1}$ Dyck words corresponding to each such cyclic composition.
It follows that $w_{2k+1, k, m} = \binom{2k+1}{k-1} N_{k,m}$ as desired.
\end{proof}

From the proofs of Lemmas ~\ref{lemma:risecompsize} and ~\ref{lemma:part1_path}, we implicitly obtain a bijection that demonstrates the symmetry $w_{2k+1, k, m} = w_{2k+1, k, k+1-m}$.
For the sake of completeness, we explicitly write out the bijection that we obtain and give an example in Figure ~\ref{im:bijection}.

\begin{dfn} \label{dfn:bijection}
Let $\pi$ be a Dyck word with semilength $2k+1$, $k$ $UD$-factors, and $m$ $UUD$-factors. 
Construct a Dyck word $\pi'$ with semilength $2k+1$, $k$ $UD$-factors, and $k+1-m$ $UUD$-factors via the following algorithm:
\begin{enumerate}
    \item Set $M$ to be the marked necklace of extended peaks associated to $\pi$ via Lemma ~\ref{lemma:risecompsize}.
    \item Decompose $M$ into a pair consisting of its underlying unmarked necklace of extended peaks $N$ and a $(k-1)$-subset $S$ of $[2k+1] = \{1,2,\dots,2k+1\}$ containing the positions of the $U$'s marked in $M$.
    \item Set $P$ to be the Dyck path of semilength $k$ with $m$ $UD$-factors that is associated to $N$ via the bijective map in Lemma ~\ref{lemma:part1_path}.
    \item Set $P'$ to be a Dyck path of semilength $k$ with $k+1-m$ $UD$-factors obtained via any bijection demonstrating the Narayana symmetry \textup{(}see \cite{Kreweras1970, Kreweras1972, Lalanne1992} for example\textup{)}. Perhaps the simplest and most intuitive is the one in terms of non-crossing set partitions in \cite{Kreweras1972}, where the author proved the stronger fact that the lattice of non-crossing partitions is self-dual.
    \item Set $N'$ to be the necklace of extended peaks associated to $P'$ via Lemma~\ref{lemma:part1_path}.
    \item  Set $M'$ to be the marked necklace of extended peaks obtained from the necklace $N'$ and subset $S$.
    \item Set $\pi'$ to be the Dyck word with semilength $2k+1$, $k$ $UD$-factors, and $k+1-m$ $UUD$-factors associated with $M'$ via Lemma~\ref{lemma:risecompsize}.
\end{enumerate}
\end{dfn}

\begin{rmk}
In Steps \textup{(}2\textup{)} and \textup{(}6\textup{)}, there is some choice of how to label the positions of the $U$'s in a necklace $N\in \neck_{2k+1,k}$ such that one can pass from a marked necklace to a pair consisting of its underlying unmarked necklace and a $(k-1)$-subset of $[2k+1]$ and vice versa.
We detail a choice of labeling that we deem to be canonical.
By Lemma ~\ref{lemma:cycle}, there is a unique ordering $(S_1, S_2, \ldots, S_k)$ of the extended peaks in $N$ such that $U^{\mu_{1}-1}DU^{\mu_{2}-1}D\cdots U^{\mu_{k}-1}D$ is $1$-dominating where $\mu_i = \ell(N_{i})$. 
Starting from the leftmost $U$, label the $U$'s in $N_1$ with the numbers $1, 2, \ldots, \mu_1$, label the $U$'s in $N_2$ with the numbers $\mu_1+1, \ldots, \mu_1+\mu_2$, and so on.
\end{rmk}

\begin{figure}[ht]
\begin{center}
\begin{subfigure}[c]{.34\textwidth}
    \begin{tikzpicture}[scale=.3, roundnode/.style={circle,fill=black!60, inner sep=1.5pt, minimum width=4pt}, triangnode/.style={regular polygon,regular polygon sides=3, fill=black!60, inner sep=1.5pt, outer sep = 0pt}]
        \draw[step=1.0, gray!100, thin] (0,0) grid (18, 3);
	\draw[black!80, line width=2pt] (0,0) -- (1,1) -- (2,2) -- (3,3) -- (4,2) -- (5,1) -- (6,2) -- (7,1) -- (8,0) -- (9,1) -- (10,2) -- (11,3) -- (12,2) -- (13,1) -- (14,0) -- (15,1) -- (16,2) -- (17,1) -- (18,0);
    \end{tikzpicture}
\end{subfigure}
 $\longrightarrow$ \hspace{.1cm}
\begin{subfigure}[c]{.3\textwidth}
$\big[\color{red}\underline{U} \,\underline{U}\color{black}UD,UD,\color{red}\underline{U}\color{black}UUD,UUD\big]$
\end{subfigure}
 $\longrightarrow$ 
 \\[.1cm]
 \hspace{.1cm}
\begin{subfigure}[c]{.35\textwidth}
    $\Big (\big[U^3D,U^2D,U^3D,UD\big], \{1,6,7\}\Big)$
\end{subfigure}
\hspace{.1cm}
$\longrightarrow$
\begin{subfigure}[c]{.3\textwidth}
    \begin{tikzpicture}[scale = .3, auto=center]
        \node at (-1, 1.5) {$\Big($};
	\draw[step=1.0, gray!100, thin] (0,0) grid (8, 3);
	\draw[black!80, line width=2pt] (0,0) -- (1,1) -- (2,0) -- (3,1) -- (4,0) -- (5,1) -- (6,2) -- (7,1) -- (8,0);
        \node at (8.3, 1) {,};
        \node at (11.4, 1.5) {$\{1, 6, 7\} \Big)$};
    \end{tikzpicture}
\end{subfigure}
$\longrightarrow$
\\[.1cm]
\begin{subfigure}[c]{.3\textwidth}
    \begin{tikzpicture}[scale = .3, auto=center]
        \node at (-1, 1.5) {$\Big($};
	\draw[step=1.0, gray!100, thin] (0,0) grid (8, 3);
	\draw[black!80, line width=2pt] (0,0) -- (1,1) -- (2,2) -- (3,3) -- (4,2) -- (5,1) -- (6,0) -- (7,1) -- (8,0);
        \node at (8.3, 1) {,};
        \node at (11.4, 1.5) {$\{1, 6, 7\} \Big)$};
    \end{tikzpicture}
\end{subfigure}
$\longrightarrow$
\hspace{.1cm}
\begin{subfigure}[c]{.35\textwidth}
    $\Big(\big[U^5D,UD,UD,U^2D\big], \{1, 6, 7\}\Big)$
\end{subfigure}
$\longrightarrow$
\\[.3cm]
\hspace{.1cm}
\begin{subfigure}[c]{.26\textwidth}
\big[\color{red}\underline{U}\color{black}UUUUD,\color{red}\underline{U}\color{black}D,\color{red}\underline{U}\color{black},UUD\big]
\end{subfigure}
$\longrightarrow$ \hspace{.3cm}
\begin{subfigure}[c]{.37\textwidth}
    \begin{tikzpicture}[scale=.33, roundnode/.style={circle,fill=black!60, inner sep=1.5pt, minimum width=4pt}, triangnode/.style={regular polygon,regular polygon sides=3, fill=black!60, inner sep=1.5pt, outer sep = 0pt}]
        \draw[step=1.0, gray!100, thin] (0,0) grid (18, 5);
	\draw[black!80, line width=2pt] (0,0) -- (1,1) -- (2,2) -- (3,3) -- (4,4) -- (5,5) -- (6,4) -- (7,3) -- (8,2) -- (9,1) -- (10,0) -- (11,1) -- (12,0) -- (13,1) -- (14,0) -- (15,1) -- (16,2) -- (17,1) -- (18,0);
    \end{tikzpicture}
\end{subfigure}
\end{center}
    \caption{Example of the bijection given in Definition ~\ref{dfn:bijection} for $k = 4$ where the Lalanne-Kreweras involution \cite{Kreweras1970, Lalanne1992} is used in Step (4). }
    \label{im:bijection}
\end{figure}

\subsection{Combinatorial proof of a related symmetry}

In addition to the symmetry in Theorem~\ref{th:main}, it can also be observed that $w_{2k-1, k, m} = w_{2k-1, k, k-m}$ for all $1 \leq m \leq k$, which is a consequence of the following variation of Theorem~\ref{th:relationNandW}.

\begin{thm}
For all $k \geq 1$ and $m \geq 0$, we have
\label{th:relationNandW2}
\begin{equation}
\label{eq:relNW2}
w_{2k-1, k, m} = \binom{2k-1}{k-1} N_{k-1,m}.    
\end{equation}
\end{thm}

Theorem~\ref{th:relationNandW2} can be proven in a way that is completely analogous to our combinatorial proof of Theorem~\ref{th:relationNandW}, but relying on Lemma \ref{lemma:part1_path2} and Corollary \ref{cor:part2_trees2k-1} below. 

\begin{lemma}
\label{lemma:part1_path2}
Let $k\geq 1$ and $m \geq 0$. Then the Narayana number $N_{k-1,m}$ is the number of cyclic compositions of $2k-1$ into $k$ parts such that exactly $m$ parts are at least 2.
\end{lemma}    

\begin{proof}
We follow the proof of Lemma \ref{lemma:part1_path} closely.
Given a cyclic composition $[\mu_{1},\mu_{2},\ldots, \mu_{k}]$ of $2k-1$ with exactly $m$ parts that are at least $2$, we build a sequence consisting of $k-1$ copies of $U$ and $k$ copies of $D$ in the same way as in the proof of Lemma \ref{lemma:part1_path}. 
By Corollary \ref{cor:reverse_cycle}, there is exactly one cyclic shift of this sequence such that any proper prefix of the sequence contains at least as many $U$'s as the number of $D$'s. 
Then the last entry of this cyclic shift is a $D$; removing this last $D$, we obtain a Dyck path of semilength $k-1$ and exactly $m$ $UD$-factors.

Conversely, consider a Dyck path of semilength $k-1$ with exactly $m$ $UD$-factors. 
We append a $D$ to the corresponding Dyck word, and form the sequence $a_1, a_2, \dots, a_k$, where $a_i$ is the number of $U$s that immediately precede the $i$th $D$.
We then add $1$ to every number in this sequence and take the equivalence class of its cyclic shifts, yielding a cyclic composition of $2k-1$ into $k$ parts, exactly $m$ of which are at least $2$.
\end{proof}

As $2k-1$ and $k$ are relatively prime, Lemma ~\ref{lemma:risecompsize} gives us the analogous corollary to Corollary ~\ref{cor:part2_trees}.

\begin{cor}
\label{cor:part2_trees2k-1}
Let $k\geq 1$. Given a cyclic composition $[\mu] \in \ccomp_{2k-1,k}$, there are exactly $\binom{2k-1}{k-1}$ Dyck words whose rise composition belongs to $[\mu]$.
\end{cor}

Using Lemma ~\ref{lemma:risecompsize} and the proofs of Lemmas ~\ref{lemma:part1_path} and \ref{lemma:part1_path2}, we obtain a proof of Theorem ~\ref{c-evenmoretantalizing} regarding a symmetry on the numbers $w_{n,k_1, k_2, \ldots, k_r}$.

\begin{proof} [Proof of Theorem \ref{c-evenmoretantalizing}]
Let $\mu \in \ccomp_{rk+1, k, k, \ldots, k, m}$. From Lemma ~\ref{lemma:risecompsize}, we have $D[\mu] = \binom{rk+1}{k-1}$ as $[\mu]$ must be primitive. Note that $\lvert \ccomp_{rk+1, k, k, \ldots, k, m} \rvert = N_{k,m}$ which can be shown via an argument analogous to that in the proof of Lemma ~\ref{lemma:part1_path}. Thus, we have $w_{rk+1, k, k, \ldots, k, m} = \binom{n}{k-1}N_{k,m}$, which gives the desired symmetry in light of the Narayana symmetry.
The symmetry for the numbers $w_{rk-1, k, k, \ldots, k, m}$ can be proven similarly.
\end{proof}



\section{Further generalizations and applications} \label{s-generalized} We now detail several interesting generalizations and applications that can be obtained from our results in Section ~\ref{combinatorial_proofs}.

\subsection{Catalan identity}
First, we obtain a formula for the Catalan numbers $C_n$ in terms of primitive cyclic compositions via Lemma ~\ref{lemma:risecompsize}.

\begin{cor}
For all $n\geq 1$, we have
\begin{equation*}
    C_n = \frac{1}{n+1} \binom{2n}{n} = \sum_{d|n} \sum_{\substack{\mathcal{[\mu]} \in \ccomp_{n/d}\\
    \mathrm{primitive}}} \frac{1}{d}\binom{n}{\ord[\mu]\cdot d-1},
\end{equation*}
where $\ccomp_{n}$ is the set of all cyclic compositions of $n$.
\end{cor}
\begin{proof}
Grouping Dyck paths by their cyclic rise composition, we have 
\begin{equation*}
C_{n} = \sum_{[\mu]\in \ccomp_{n}} D[\mu].
\end{equation*}
Recall that every cyclic composition of $n$ can be uniquely expressed as the concatenation of $d$ copies of a primitive cyclic composition of $n/d$ for some divisor $d$ of $n$. 
Similarly, for every divisor $d$ of $n$, each primitive cyclic composition of $n/d$ can be made into a cyclic composition of $n$ by concatenating $d$ copies. 
This gives us
\begin{equation*}
C_n = \sum_{[\mu]\in \ccomp_{n}} D[\mu] = \sum_{d|n}\sum_{\substack{[\mu] \in \ccomp_{n/d}\\ \text{primitive}}} D([\mu]^d)
\end{equation*}
where $[\mu]^d$ is the concatenation of $d$ copies of $[\mu]$.
From Lemma ~\ref{lemma:risecompsize}, this gives us precisely
\begin{equation*}
C_n = \sum_{d|n}\sum_{\substack{[\mu] \in \ccomp_{n/d}\\ \text{primitive}}} D([\mu]^d) = \sum_{d|n} \sum_{\substack{\mathcal{[\mu]} \in \ccomp_{n/d}\\
    \text{primitive}}} \frac{1}{d}\binom{n}{\ord[\mu]\cdot d-1}. \qedhere
\end{equation*}
\end{proof}

\subsection{Generalized Narayana identities}
Next, we generalize Theorems ~\ref{th:relationNandW} and \ref{th:relationNandW2} by expressing the numbers $w_{n,k,m}$, for any $n\neq 2k$, in terms of a family of generalized Narayana numbers due to Callan ~\cite{callan}.

Given $0 \leq r \leq n$ and $0 \leq k \leq n-r$, define the \emph{$r$-generalized Narayana number} $N_{n, k}^{(r)}$ by
\[N_{n, k}^{(r)} = \frac{r+1}{n+1} \binom{n+1}{k} \binom{n-r-1}{k-1}.\]
Observe that the usual Narayana numbers $N_{n,k}$ can be obtained by setting $r = 0$ in $N_{n, k}^{(r)}$.
For $k <0$, we use the convention that $\binom{n}{k} = 0$ except for the special case when $n = k = -1$, where we define $\binom{-1}{-1}$ to be $1$. 

The following is a generalization of Theorems \ref{th:relationNandW} and \ref{th:relationNandW2}.

\begin{thm}
 \label{thm:simplify}
For all $j,k \geq 1$ and $m\geq 0$, we have
\begin{equation*}
    w_{2k+j,k,m} = \frac{1}{j}\binom{2k+j}{k-1}N^{(j-1)}_{k+j-1,m}
\end{equation*}
and for all $1 \leq j \leq k$ and $m\geq 0$, we have
\begin{equation*}
    w_{2k-j,k,m} = \frac{1}{j}\binom{2k-j}{k-1}N^{(j-1)}_{k-1,m}.
\end{equation*}
\end{thm}

Before proving Theorem ~\ref{thm:simplify}, we first introduce a generalization of Dyck paths and prove a useful lemma.
Consider paths in $\mathbb{Z}^2$ from $(0,0)$ to $(2n-r, r)$, consisting of $n$ up steps $(1,1)$ and $n-r$ down steps $(1,-1)$, that never pass below the horizontal axis.
Denote by $D^{(r)}_{n,k}$ the set of words on the alphabet $\{U, D \}$ corresponding to such paths with exactly $k$ $UD$-factors. 
In~\cite{callan}, Callan describes a proof by Schulte showing that $N_{n, k}^{(r)}$ is the cardinality of $D^{(r)}_{n,k}$.

For $\omega, \nu \in D^{(r)}_{n,k}$, let us write $\omega \sim \nu$ if the words $U \omega$ and $U \nu$ are cyclic shifts of each other. 
The relation $\sim$ is an equivalence relation on $D^{(r)}_{n,k}$, and we denote the set of its equivalence classes by $\tilde{D}^{(r)}_{n,k}$. 
For $[\omega] \in \tilde{D}^{(r)}_{n,k}$, let $\ord[\omega]$ be the number of distinct elements of $D^{(r)}_{n,k}$ contained within the equivalence class $[\omega]$.

\begin{dfn}
For $j,k \geq 1$, let $\phi_{j,k}$ be the map from $\ccomp_{2k+j, k, m}$ to $\tilde{D}^{(j-1)}_{k+j-1, m}$ where $\phi_{j,k}[\mu]$ is obtained via the following algorithm:
\begin{enumerate}
    \item For $[\mu] = [\mu_{1}, \ldots, \mu_{k}] \in \ccomp_{2k+j, k,m}$, set $\omega = U^{\mu_{1}-1}DU^{\mu_{2}-1}D \cdots U^{\mu_{k}-1}D$.
    \item Let $\nu = \nu_{1} \nu_{2} \cdots \nu_{2k+j}$ be any cyclic shift of $\omega$ that is $1$-dominating.
    \item Set $\phi_{j,k}[\mu]$ to be the equivalence class of the subword $\nu_{2} \cdots \nu_{2k+j}$.
\end{enumerate}
\end{dfn}

It is not immediately clear from the above definition whether the map $\phi_{j,k}$ is well-defined, but this will be established in the proof of the following lemma.

\begin{lemma} \label{lemma:genbijection}
For all $j,k \geq 1$, the map $\phi_{j,k}$ is a bijection. Moreover, for all $[\mu] \in \ccomp_{2k+j,k,m}$, we have
\begin{equation*}
    \frac{\ord(\phi_{j,k}[\mu])}{\ord[\mu]} = \frac{j}{k}.
\end{equation*}
\end{lemma}
\begin{proof}
We first prove that $\phi_{j,k}$ is well-defined. 
Since $\omega$ contains $k+j$ copies of $U$ and $k$ copies of $D$, the cycle lemma guarantees that at least one cyclic shift of $\omega$ is $1$-dominating. 
By construction of $\omega$, the word $\nu$ must contain exactly $m$ $UD$-factors. 
The fact that $\nu$ is $1$-dominating and contains $m$ $UD$-factors implies that its subword $\nu_{2} \cdots \nu_{2k+j}$ is an element of $D^{(j-1)}_{k+j-1, m}$.
From the definition of $\sim$ on $D^{(j-1)}_{k+j-1, m}$, any $1$-dominating cyclic shift of $\omega$ will be sent to the same equivalence class in $\tilde{D}^{(j-1)}_{k+j-1, m}$. 
This same argument also implies that $\phi_{j,k}[\mu]$ does not depend on the representative of $[\mu]$ that is chosen.

Injectivity and surjectivity are straightforward to check from the definition of $\phi_{j,k}$.
By the cycle lemma, there are exactly $j$ cyclic shifts of $\omega$ that are $1$-dominating; among these $j$ words, there are $j \cdot \ord[\mu]/k$ \emph{distinct} cyclic shifts as each of them appears $k/\ord[\mu]$ times.
These $1$-dominating sequences are in bijection with paths in the equivalence class of $\phi_{j,k}[\mu]$ by removing the first $U$ from the sequence. 
Thus, $\ord(\phi_{j,k}[\mu]) = j \cdot \ord[\mu]/k$.
\end{proof}

We are now ready to prove Theorem ~\ref{thm:simplify}.

\begin{proof} [Proof of Theorem ~\ref{thm:simplify}]
From Lemma ~\ref{lemma:genbijection}, we have 
\begin{equation*}
\lvert \comp_{2k+j, k, m} \rvert = \frac{k}{j}\lvert D^{(j-1)}_{k+j-1,m} \rvert = \frac{k}{j} N^{(j-1)}_{k+j-1,m}.
\end{equation*}
Substituting this into ~\eqref{e-wcomp} gives the desired result
\begin{equation*}
    w_{2k+j, k, m} = \frac{1}{j} \binom{2k+j}{k-1} N^{(j-1)}_{k+j-1,m}.
\end{equation*}

The proof for $w_{2k-j,k,m}$ follows similarly by defining an analogous map $\varphi_{j,k}$ from $\ccomp_{2k-j, k, m}$ to $\tilde{D}^{(j-1)}_{k-1, m}$ and reproving Lemma ~\ref{lemma:genbijection} for $\varphi_{j,k}$.
\end{proof}

Lemma ~\ref{lemma:genbijection} naturally leads to the following generalization of Lemmas ~\ref{lemma:part1_path} and ~\ref{lemma:part1_path2}, which expresses the number of cyclic compositions in $\ccomp_{2k\pm j, k,m}$ in terms of $r$-generalized Narayana numbers. 
Below, $\varphi$ denotes Euler's totient function.

\begin{prop} \label{p-ccomp}
    Let $k \geq 1$ and $m,j\geq 0$, and let $d = \mathsf{gcd}(k,m,j)$.\footnote{If $m=0$ or $j=0$, then $\mathsf{gcd}(k,m,j)$ is defined to be the greatest common divisor of the nonzero numbers among $k$, $m$, and $j$.}\vspace{5bp}
    \begin{enumerate}
    \item[(a)] If $j \geq 1$, we have $\displaystyle{\lvert \ccomp_{2k+j, k, m} \rvert = \frac{1}{j} \sum_{s\mid d} \varphi(s)N^{(j/s-1)}_{(k+j)/s-1, \, m/s}}$.\vspace{5bp}
    \item[(b)] If $j=0$, we have $\displaystyle{\lvert \ccomp_{2k, k, m} \rvert = \frac{1}{k} \sum_{s\mid d} \varphi(s) \binom{\frac{k}{s}-1}{\frac{m}{s}-1}\binom{\frac{k}{s}}{\frac{m}{s}}}$.
    \vspace{5bp}
    \item[(c)] If $1 \leq j \leq k$, we have $\displaystyle{\lvert \ccomp_{2k-j, k, m} \rvert = \frac{1}{j} \sum_{s\mid d} \varphi(s)N^{(j/s-1)}_{k/s-1,\, m/s}}$.
    \end{enumerate}
\end{prop}
\begin{proof}
    Let us call an (ordinary) composition $\mu$ \emph{primitive} if $[\mu]$ is primitive, and let $\pcomp_{n,k,m}$ denote the set of primitive compositions of $n$ with $k$ parts with exactly $m$ parts at least two.
    Let $1\leq k \leq m$ and $j\geq 0$, and let $d = \mathsf{gcd}(k,m,j)$.
    Given $\ell \mid d$, define
    $$f(\ell) = \lvert \comp_{(2k+j)\ell/d, \, k\ell/d, \, m\ell/ d} \rvert \quad \text{and} \quad g(\ell) = \lvert \pcomp_{(2k+j)\ell/d,\, k\ell/d,\, m\ell /d} \rvert.$$
    Every composition can be uniquely decomposed as a concatenation of one or more copies of a primitive composition, which leads to the formula $f(\ell) = \sum_{s\mid \ell} g(s)$.
    By M\"obius inversion, we then have $g(\ell) = \sum_{s\mid \ell} \mob(s) f(\ell/s)$ where $\mob$ is the M\"obius function.
    Observe that 
    \begin{equation*}
    \lvert \ccomp_{2k+j, k, m} \rvert = \sum_{\ell\mid d} \frac{\ell}{k} \lvert \pcomp_{(2k+j)/\ell, \, k/\ell, \, m/\ell} \rvert;
    \end{equation*}
    after all, every cyclic composition in $\ccomp_{2k+j, k, m}$ is a concatenation of $\ell$ copies of a primitive cyclic composition with $k/\ell$ parts for some $\ell$ dividing $d$, and this primitive cyclic composition is the cyclic equivalence class of $k/\ell$ elements of $\pcomp_{(2k+j)/\ell, \,k/\ell , \,m/\ell}$. We then have
    {\allowdisplaybreaks\begin{align*}
    \lvert \ccomp_{2k+j, k, m} \rvert &= \sum_{\ell\mid d} \frac{\ell}{k} \lvert \pcomp_{(2k+j)/\ell, \, k/\ell, \, m/\ell} \rvert \\
    &= \sum_{\ell\mid d} \frac{\ell}{k} g\Big(\frac{d}{\ell}\Big) \\
    &= \frac{1}{k}\sum_{\ell\mid d} \sum_{q\mid (d/\ell)} \mob(q) \ell f\Big(\frac{d}{\ell q} \Big) \\
    &= \frac{1}{k}\sum_{s\mid d} \sum_{\ell q = s} \mob(q) \frac{s}{q} f\Big(\frac{d}{s}\Big). \\
    &= \frac{1}{k}\sum_{s\mid d} \varphi(s) f\Big(\frac{d}{s}\Big),
    \end{align*}}where the last step uses the well-known identity $\varphi(s) = \sum_{q\mid s} \mob(q) s/q$.
    If $j\geq 1$, then we have 
    $$f\Big(\frac{d}{s} \Big) = \lvert \comp_{2(k/s)+j/s,\, k/s,\, m/s} \rvert = \frac{k}{j} N^{(j/s-1)}_{(k+j)/s-1,\, m/s}$$ 
    by Lemma ~\ref{lemma:genbijection}, and if $j=0$, then we instead have
    $$f\Big(\frac{d}{s} \Big) = \lvert \comp_{2(k/s),\, k/s,\, m/s} \rvert = \binom{\frac{k}{s}-1}{\frac{m}{s}-1}\binom{\frac{k}{s}}{\frac{m}{s}}$$
    by \eqref{e-compbin}; substituting appropriately completes the proof of parts (a) and (b). 
    We omit the proof of (c) as it is similar to that of (a).
\end{proof}

\begin{rmk}
    Proposition ~\ref{p-ccomp} has an interesting interpretation related to permutation enumeration, as $\lvert \ccomp_{n, k, m} \rvert$ is the number of distinct cyclic descent sets among cyclic permutations of length $n$ with $k$ cyclic descents and $m$ cyclic peaks \textup{(}for all $1 \leq k<n$\textup{)}; see \cite{sagan2, sagan1, liang, sagan3} for definitions.
    In particular, when $j\neq 0$ and $\mathsf{gcd}(k,j,m) = 1$, the number of cyclic descent classes among such cyclic permutations of length $2k+j$ is equal to a generalized Narayana number divided by $j$.
    The case $j=\pm 1$ \textup{(}Lemmas ~\ref{lemma:part1_path} and ~\ref{lemma:part1_path2}\textup{)} yields a new interpretation of the \textup{(}ordinary\textup{)} Narayana numbers $N_{k,m}$ in terms of cyclic descent classes.
\end{rmk}


\section{Polynomials}\label{polynomials}


\subsection{Real-rootedness}
A natural question is whether or not the sequence $\{ w_{n,k,m} \}_{0 \leq m \leq k}$, for a fixed $n$ and $k$, is unimodal. 
In other words, for fixed $n$ and $k$, does there always exist $0 \leq j \leq k$ such that 
$$w_{n,k,0} \leq w_{n,k, 1} \leq \cdots \leq w_{n,k,j} \geq w_{n, k, j+1} \geq \cdots \geq w_{n,k,k}?$$
One way to prove unimodality results in combinatorics is through real-rootedness. 
A polynomial with coefficients in $\mathbb{R}$ is said to be \emph{real-rooted} if all of its roots are in $\mathbb{R}$. 
(We use the convention that constant polynomials are also real-rooted.)
It is well known that if a polynomial with non-negative coefficients is real-rooted, then the sequence of its coefficients is unimodal (see ~\cite{Branden2015}, for example). 

Let $W_{n,k}(t)$ be the polynomial defined by
\begin{align*}
    W_{n,k}(t) = \sum_{m=0}^{k} w_{n,k,m}t^m.
\end{align*}
In what follows, we prove that the polynomials $W_{n,k}(t)$ are real-rooted, thus implying the unimodality of the sequences $\{ w_{n,k,m} \}_{0 \leq m \leq k}$.

We begin with a simple result involving the roots of $W_{n,k}(t)$.

\begin{prop}\label{prop:reflection}
For all $1 \leq k \leq n-1$, the polynomials $W_{n,k}(t)$ and $W_{n,n-k}(t)$ have the same roots. 
\end{prop}
\begin{proof}
This follows from the fact that $w_{n,k,m}= \frac{k(k+1)}{(n-k)(n-k+1)} w_{n,n-k,m}$, which is readily verified from Theorem ~\ref{t-explicit}.
\end{proof}

To prove the real-rootedness of the $W_{n,k}(t)$, we make use of Malo's result regarding the roots of the Hadamard product of two real-rooted polynomials.

\begin{thm}[\cite{malo}]\label{thm:Malo} 
Let $f(t) = \sum_{i = 0}^{m} a_{i} t^{i}$ and $g(t) = \sum_{i=0}^{n} b_{i} t^{i}$ be real-rooted polynomials in $\mathbb{R}[t]$ such that all the roots of $g$ have the same sign. 
Then their Hadamard product
$$f\ast g = \sum_{i = 0}^{\ell} a_{i}b_{i} t^{i},$$ 
where $\ell = \min\{m,n\}$, is real-rooted.
\end{thm}

\begin{thm}\label{thm:real-rooted}
For all $n, k \geq 0$, the polynomials $W_{n,k}(t)$ are real-rooted.
\end{thm}
\begin{proof}
From Theorem ~\ref{t-explicit}, we have
\begin{equation}
W_{n,k}(t) = \begin{cases}
    0, & \text{ if } n < k,\\
    1, & \text{ if } n = k,\\
    \frac{1}{k}\binom{n}{k-1}\sum_{m = 1}^{\text{min}\{k, n-k \}} \binom{n-k-1}{m-1} \binom{k}{m} t^m, & \text{ if } n > k.
\end{cases}
\end{equation}
Thus it suffices to check that the polynomial $\sum_{m = 1}^{\text{min}\{k, n-k\}} \binom{n-k-1}{m-1}\binom{k}{m} t^m$ is real-rooted, which follows from applying Theorem ~\ref{thm:Malo} to $f(t) = t(t-1)^{n-k-1}$ and $g(t) = (t-1)^k$.
\end{proof}

More generally, we conjecture the polynomials $W_{n,k}(t)$ satisfy stronger conditions which we presently define. 
For two real-rooted polynomials $f$ and $g$, let $\{u_i\}$ be the roots of $f$ and $\{v_i\}$ the roots of $g$, both in non-increasing order. We say that $g$ \emph{interlaces} $f$, denoted by $g \rightarrow f$, if either $\deg(f) = \deg(g)+1  = d$ and 
$$u_{d} \leq v_{d-1} \leq u_{d-1} \leq \cdots \leq v_{1} \leq u_{1},$$
or if $\deg(f) = \deg(g)  = d$ and 
$$v_{d} \leq u_{d} \leq v_{d-1} \leq u_{d-1} \leq \cdots \leq v_{1} \leq u_{1}.$$
(By convention, we assume that a constant polynomial interlaces with every real-rooted polynomial.)
We say that a sequence of real-rooted polynomials $f_{1}, f_{2}, \ldots$ is a \emph{Sturm sequence} if $f_{1} \rightarrow f_{2} \rightarrow \cdots.$
Moreover, a finite sequence of real-rooted polynomials $f_{1}, f_{2}, \ldots, f_{n}$ is said to be \emph{Sturm-unimodal} if there exists $1 \leq j \leq n$ such that 
$$f_{1} \rightarrow f_{2} \rightarrow \cdots \rightarrow\ f_{j} \leftarrow f_{j+1} \leftarrow \cdots \leftarrow f_{n}.$$

\begin{conj}
    For any fixed $k \geq 1$, the polynomials $\{W_{n,k}(t)\}_{n \geq k}$ form a Sturm sequence.
\end{conj}

\begin{conj}
    For any fixed $n \geq 1$, the sequence $\{W_{n,k}(t)\}_{1\leq k \leq n}$ is Sturm-unimodal.
\end{conj}
\noindent We note that Conjectures 5.4 and 5.5 were originally posed in previous versions of this paper; since their appearance, these conjectures have been resolved in \cite{WangZhang}.

Our result expressing the numbers $w_{n,k,m}$ in terms of generalized Narayana numbers has a natural polynomial analogue. Let $\operatorname{Nar}^{(r)}_k(t)$ denote the $k$th \textit{$r$-generalized Narayana polynomial} defined by
$$\operatorname{Nar}^{(r)}_k(t)=\sum_{m=0}^{k-r}N^{(r)}_{k,m}t^m=\frac{r+1}{k+1}\sum_{m=0}^{k-r}\binom{k+1}{m}\binom{k-r-1}{m-1}t^m.$$
Setting $r=0$ recovers the usual \emph{Narayana polynomials} $\operatorname{Nar}_{k}(t) = \sum_{m = 0}^{k} N_{k,m} t^m$. 
From Theorem ~\ref{thm:simplify} and straightforward computations, we have the following expressions for $W_{n,k}(t)$.

\begin{prop} \label{prop:polysimplify}
Let $k\geq 1$.\vspace{5bp}
\begin{enumerate}
    \item[(a)] For all $j\geq1$, we have $\displaystyle{W_{2k+j,k}(t)= \frac{1}{j}\binom{2k+j}{k-1}\operatorname{Nar}^{(j-1)}_{k+j-1}(t)}$.\vspace{5bp}
    \item[(b)] We have $\displaystyle{W_{2k,k}(t) = C_k \sum_{m=1}^{k}\binom{k-1}{m-1}\binom{k}{m}t^m}$ where $C_{k}$ denotes the $k$th Catalan number.\vspace{5bp}
    \item[(c)] For all $1\leq j\leq k$, we have $\displaystyle{W_{2k-j,k}(t)= \frac{1}{j}\binom{2k-j}{k-1}\operatorname{Nar}^{(j-1)}_{k-1}(t)}$.
\end{enumerate}
\end{prop}

The real-rootedness of the $r$-generalized Narayana polynomials was recently shown in \cite{chen-yang-zhao} using a different approach; Proposition \ref{prop:polysimplify} shows that the real-rootedness of the $W_{n,k}(t)$ implies the real-rootedness of  the $\operatorname{Nar}^{(r)}_{k}(t)$, thus giving an alternative proof of this result.


\subsection{Symmetry, \texorpdfstring{$\gamma$}{gamma}-positivity, and a symmetric decomposition} It is fitting that we end this paper by returning full circle to the topic of symmetry. A polynomial $a_0 + a_1t+\cdots+a_{d}t^d$ of degree $d$ is said to be \emph{symmetric} if $a_i = a_{d-i}$ for all $0\leq i \leq d$. First, we note that our symmetries for the numbers $w_{2k+1,k,m}$ and $w_{2k-1,k,m}$ immediately imply the following:

\begin{prop}\label{prop:symmetric}
The polynomials $W_{2k+1,k}(t)$ and $W_{2k-1,k}(t)$ are symmetric.
\end{prop}

A symmetric polynomial of degree $d$ can be written uniquely as a linear combination of the polynomials $\{ t^j(1+t)^{d-2j} \}_{0 \leq j \leq \lfloor d/2 \rfloor}$, referred to as the \emph{gamma basis}. A symmetric polynomial is called \emph{$\gamma$-positive} if its coefficients in the gamma basis are nonnegative. Gamma-positivity has shown up in many combinatorial and geometric contexts; see \cite{Athanasiadis2018} for a thorough survey. It is well known that the coefficients of a $\gamma$-positive polynomial form a unimodal sequence, and that $\gamma$-positivity is connected to real-rootedness in the following manner:

\begin{thm}[\cite{Branden2004}] \label{thm:gammapositive}
    If $f$ is a real-rooted, symmetric polynomial with nonnegative coefficients, then $f$ is $\gamma$-positive.
\end{thm}

The $\gamma$-positivity of the polynomials $W_{2k+1,k}(t)$ and $W_{2k-1,k}(t)$ then follows directly from Theorem ~\ref{thm:real-rooted}, Proposition \ref{prop:symmetric}, and Theorem \ref{thm:gammapositive}. Furthermore, we can get explicit formulas for their gamma coefficients by exploiting their connection to the Narayana polynomials.

\begin{prop}\label{cor:gamma-positivity}
The polynomials $W_{2k+1,k}(t)$ and $W_{2k-1,k}(t)$ are $\gamma$-positive for all $k\geq 1$. More precisely, we have the following gamma expansions:
\vspace{5bp}
\begin{enumerate}
\item[(a)] $\displaystyle{W_{2k+1,k}(t)=\sum_{j=1}^{\lfloor \frac{k+1}{2}\rfloor} \binom{2k+1}{k-1} \frac{(k-1)!}{(k-2j+1)!\, (j-1)!\, j!} \, t^j(1+t)^{k+1-2j}}$ for all $k\geq 1$\textup{;} \vspace{5bp}
\item[(b)] $\displaystyle{W_{2k-1,k}(t)=\sum_{j=1}^{\lfloor \frac{k+1}{2}\rfloor} \binom{2k-1}{k-1} \frac{(k-2)!}{(k-2j)!\, (j-1)!\, j!} \, t^j(1+t)^{k+1-2j}}$ for all $k\geq 2$.
\end{enumerate}
\end{prop}

\begin{proof}
The Narayana polynomials are known to be $\gamma$-positive with gamma expansion
$$\operatorname{Nar}_k (t)=\sum_{j=1}^{\lfloor \frac{k+1}{2} \rfloor}  \frac{(k-1)!}{(k-2j+1)!\, (j-1)!\, j!} \, t^j(1+t)^{k+1-2j}$$
for all $k\geq 1$ \cite[Theorem 2.32]{Athanasiadis2018}, and this implies the desired result by the $j=1$ case of Proposition ~\ref{prop:polysimplify} (a) and (c).
\end{proof}

\begin{rmk} \label{r-gam}
The gamma coefficients of the Narayana polynomials have a nice combinatorial interpretation in terms of lattice paths: $\frac{(k-1)!}{(k-2j+1)!\, (j-1)!\, j!}$ is the number of Motzkin paths of length $k-1$ with $j-1$ up steps \cite{blanco}. 
We can use this fact to give combinatorial interpretations of the gamma coefficients of $W_{2k+1,k}(t)$ and $W_{2k-1,k}(t)$. It would be interesting to find a combinatorial proof for Proposition ~\ref{cor:gamma-positivity} using Motzkin paths, perhaps in the vein of  the ``valley-hopping'' proof for the $\gamma$-positivity of Narayana polynomials \cite{hoppityhophop}.
\end{rmk}

Our question in Remark \ref{r-gam} was recently been addressed by Fu and Yang \cite{fu2025group}, who gave a combinatorial proof for Proposition ~\ref{cor:gamma-positivity} using a group action, reminiscent of valley-hopping, on cyclic compositions.

While the polynomials $W_{n,k}(t)$ are not symmetric in general, it turns out that we can always express $W_{n,k}(t)$ as the sum of two symmetric polynomials.

\begin{thm}\label{conj:symmetric_decomposition}
Let $1\leq k \leq n$. Then there exist symmetric polynomials $W_{n,k}^+(t)$ and $W_{n,k}^-(t)$, both with nonnegative coefficients, such that:
\begin{enumerate}
    \item[(a)] if $n<2k$, then $W_{n,k}(t)= W_{n,k}^+(t) -tW_{n,k}^-(t)$;
    \item[(b)] if $n=2k$, then $W_{n,k}(t)= W_{n,k}^+(t) + tW_{n,k}^-(t)$;
    \item[(c)] and if $n>2k$, then $W_{n,k}(t)= -W_{n,k}^+(t)+tW_{n,k}^-(t)$.
\end{enumerate}
\end{thm}

\begin{proof}
Since $n$ and $k$ are fixed, let us simplify notation by writing $w_m$ in place of $w_{n,k,m}$, so that $W_{n,k}(t)=\sum_{m=0}^k w_m t^m$. Let
\begin{equation}\label{w+}
    w^+_{i+1}=\Big( \sum_{j=0}^{i+1} w_j \Big) - \Big( \sum_{j=0}^{i} w_{k-j} \Big)
    \quad \text{and} \quad
    w^-_{i}= - \Big( \sum_{j=0}^{i} w_j \Big) + \Big( \sum_{j=0}^{i} w_{k-j} \Big)
\end{equation}
for all $0 \leq i \leq k$; also take $w^+_0=1$ when $k=n$ and $w^+_0=0$ otherwise. Observe that
\begin{align}
w^+_i+ w^-_{i-1} &= \Big( \sum_{j=0}^{i} w_j \Big) - \Big( \sum_{j=0}^{i-1} w_{k-j} \Big) - \Big( \sum_{j=0}^{i-1} w_j \Big) + \Big( \sum_{j=0}^{i-1} w_{k-j} \Big) = w_i, \label{eq:sum} \\
w^+_i-w^+_{k-i} &= \Big( \sum_{j=0}^{i} w_j \Big) - \Big( \sum_{j=0}^{i-1} w_{k-j} \Big) - \Big( \sum_{j=0}^{k-i} w_j \Big) + \Big( \sum_{j=0}^{k-i-1} w_{k-j} \Big) = 0, \quad \text{and} \label{eq:sym1} \\
w^-_{i-1}-w^-_{k-i} &= - \Big( \sum_{j=0}^{i-1} w_j \Big) + \Big( \sum_{j=0}^{i-1} w_{k-j} \Big) + \Big( \sum_{j=0}^{k-i} w_j \Big) - \Big( \sum_{j=0}^{k-i} w_{k-j} \Big) = 0. \label{eq:sym2}
 \end{align}
A standard induction argument utilizing the explicit formula in Theorem ~\ref{t-explicit} yields the following:
\begin{itemize}
    \item the $w_i^+$ are positive when $n\leq 2k$,
    \item the $w_i^+$ are negative when $n>2k$,
    \item the $w_i^-$ are positive when $n\geq 2k$, and 
    \item the $w_i^-$ are negative when $n<2k$.
\end{itemize}
Define the polynomials $W_{n,k}^{+}(t)$ and $W_{n,k}^{-}(t)$ by
\begin{align*}
W_{n,k}^{+}(t) = \sum_{i=0}^{k}|w_i^{+}|\, t^{i}\quad\text{and}\quad W_{n,k}^{-}(t)=\sum_{i=0}^{k}|w_i^{-}|\, t^{i},
\end{align*}
respectively. These polynomials are symmetric by (\ref{eq:sym1}) and (\ref{eq:sym2}), and the decompositions given in (a)--(c) hold by construction in light of (\ref{eq:sum}).
\end{proof}

We end with a couple of conjectures concerning the polynomials $W_{n,k}^{+}(t)$ and $W_{n,k}^{-}(t)$ arising from our symmetric decomposition. Our first conjecture, concerning real-rootedness, has been numerically verified for all $n\leq100$. 

\begin{conj}
The following completely characterizes when $W_{n,k}^{+}(t)$ and
$W_{n,k}^{-}(t)$ are both real-rooted\textup{:}
\begin{enumerate}
\item [(a)] If $n=1$ or $n=2$, then $W_{n,k}^{+}(t)$ and $W_{n,k}^{-}(t)$ are both real-rooted if and only if $k=1$.
\item [(b)] If $n>1$ and $n \equiv 1 \textup{ (mod 4)}$, then $W_{n,k}^{+}(t)$ and $W_{n,k}^{-}(t)$ are both real-rooted if and only if $k=1,2,\left\lfloor n/2\right\rfloor -1,\left\lfloor n/2\right\rfloor ,\text{or }\left\lfloor n/2\right\rfloor +1$.
\item [(c)] If $n \equiv 3 \textup{ (mod 4)}$, then $W_{n,k}^{+}(t)$ and $W_{n,k}^{-}(t)$ are both real-rooted if and only if $k=1,2,\left\lfloor n/2\right\rfloor -1,\left\lfloor n/2\right\rfloor ,\left\lfloor n/2\right\rfloor +1,\text{or }\left\lfloor n/2\right\rfloor +2$.
\item [(d)] If $n$ is even and not equal to $2$, $10$, $12$, or $16$, then $W_{n,k}^{+}(t)$ and $W_{n,k}^{-}(t)$ are both real-rooted if and only if $k=1,2,n/2-1,n/2,\text{or }n/2+1$.
\item [(e)] If $n=10$, $12$, or $16$, then $W_{n,k}^{+}(t)$ and $W_{n,k}^{-}(t)$ are both real-rooted if and only if $k=1,2,n/2-2,n/2-1,n/2,\text{or }n/2+1$.
\end{enumerate}
\end{conj}

To establish real-rootedness, it would be helpful to have formulas for the polynomials $W_{n,k}^{+}(t)$ or $W_{n,k}^{-}(t)$. Our next conjecture gives formulas for $W_{2k,k}^{+}(t)$, $W_{2k,k}^{-}(t)$, $W_{2k,k+1}^{-}(t)$, and $W_{2k,k-1}^{+}(t)$. Unfortunately, we do not have conjectured formulas for any of the other polynomials.



\begin{conj} \label{conj-W2k} 
Let $k\geq 1$.
\begin{enumerate}
\item[(a)] We have
\[
W_{2k,k}^{+}(t)=(k-1)C_{k}\Nar_{k-1}(t)\quad\text{and}\quad W_{2k,k}^{-}(t)=C_{k}\sum_{i=0}^{k-1}\binom{k-1}{i}^{2}t^{i}.
\]
\item[(b)] If $k\geq 2$, we have 
\[
W_{2k,k+1}^{-}(t)=\binom{2k}{k}\Nar_{k-1}(t) \quad \text{and} \quad
W_{2k,k-1}^{+}(t)=-\frac{t}{2}\binom{2k}{k-2}\overline{\Nar}_{k-2}^{(1)}(t)
\]
where $\overline{\Nar}_{k}^{(j)}(t)=\sum_{i=0}^{k-j}\frac{j+1}{k+1}\binom{k+1}{i}\binom{k+1}{i+j+1}t^{i}$.
\end{enumerate}
\end{conj}

The $\overline{\Nar}_{k}^{(j)}(t)$ defined in Conjecture \ref{conj-W2k} form a family of generalized Narayana polynomials \cite{yang} different from Callan's.

\section*{Acknowledgments}
The authors thank the American Mathematical Society and the organizers of the MRC on Trees in Many Contexts, where this work was initiated. We also thank David Callan, Ira Gessel, and Martin Rubey for helpful discussions, as well as several anonymous referees for carefully reading previous versions of this manuscript and kindly offering suggestions which led to significant improvements in the quality and organization of our paper. This material is based upon work supported by the National Science Foundation under Grant Number DMS-1641020.


\bibliographystyle{plain}
\bibliography{main}


\section{Appendix}
\label{sec:app}

\subsection{The generating function approach}\label{explicit_solution}

Here we present an algebraic proof of Theorem~\ref{t-explicit}, which gives an explicit formula for the numbers $w_{n,k,m}$ and in turn implies the symmetry in Theorem~\ref{th:main}. 
To do so, we need a functional equation satisfied by the generating function $W$ defined by
\[W=W(x,y,z)=\sum_{n,k,m}w_{n,k,m}\,x^ny^kz^m.\]
In other words, $W$ is the generating function of all Dyck paths, where the exponents of $x$, $y$, and $z$ encode the semilength, the number of $UD$-factors, and the number of $UUD$-factors, respectively.

\begin{lemma}
The functional equation
\begin{equation}\label{func_eq-first}
(x-x^2y+x^2yz)W^2-(1+x-xy)W+1=0
\end{equation}
holds. 
\end{lemma}

\begin{proof}
Let us call a Dyck path of positive semilength \emph{irreducible} if it does not touch the horizontal axis, other than at its starting point and endpoint.
Then each nonempty Dyck path $P$ uniquely decomposes as the concatenation of an irreducible Dyck path and a Dyck path (in that order), where the latter is the empty path if $P$ itself is irreducible. 
If $V(x,y,z)$ is the generating function for irreducible Dyck paths defined similarly to $W(x,y,z)$ but only for irreducible paths, then this leads to the functional equation
\begin{equation} \label{funceq1}
W = 1+ VW.
\end{equation}

Let $I$ be an irreducible Dyck path; then $I$ must start with an up step, end with a down step, and consist of a Dyck path $P_1$ between the two.
If $P_1$ is empty, then $I$ is simply the path $UD$, contributing $xy$ to the generating function $V(x,y,z)$.
Otherwise, $I=UP_1D$, and $I$ has the same number of $UD$-factors as $P_1$, while its semilength is one longer.
The number of $UUD$-factors of $I$ compared with that of $P_1$, however, depends on whether $P_1$ begins with a $UD$, so let us consider these two cases separately: \vspace{5bp}
\begin{itemize}
\item If $P_1$ does begin with a $UD$, then $P_1$ is of the form $UDP_2$ and $I=UP_1D=UUDP_2D$ has one more $UUD$-factor than $P_1$.
Therefore, in this case, $I$ has one more $UD$-factor and one more $UUD$-factor than $P_2$, while its semilength is two longer.
Dyck paths of this type contribute $x^2yzW$ to $V$. \vspace{5bp}
\item If $P_1$ does not begin with a $UD$, then $I$ also has the same number of $UUD$-factors as $P_1$; Dyck paths of this type contribute $x(W-1-xyW)$ to the generating function $V$. \vspace{5bp}
\end{itemize}
Thus, we have the functional equation
\begin{equation} \label{funceq2}
  V= xy+ x^2yzW + x(W-1-xyW).
  \end{equation}
The proof of our claim is now routine by substituting the expression obtained for $V$ in \eqref{funceq2} into  (\ref{funceq1}) and rearranging.  
\end{proof}

\begin{proof}[Algebraic proof of Theorem ~\ref{t-explicit}]
Equation (\ref{func_eq-first}) can be rewritten as  
\begin{equation*}\label{func_eq2}
uW^2-vW+1=0, \quad {\text{where }} u=x-x^2y+x^2yz {\text{ and }} v=1+x-xy.
\end{equation*}
Solving for $W$ in this second degree equation, after elementary manipulations, gives
\begin{equation*}\label{quad_sol}
W=\frac{v-\sqrt{v^2-4u}}{2u}=\frac{v}{2u} \left(1-\sqrt{1-4 \frac{u}{v^2}}\right). 
\end{equation*}
Setting $t=u/v^2$ in the well known classical equation
\begin{equation*}
\frac{1-\sqrt{1-4t}}{2t}=\sum_{\nu \geq 0} C_\nu t^\nu,
\end{equation*}
where $C_\nu=\binom{2\nu}{\nu}/{(\nu+1)}$ is the $\nu^{th}$ Catalan number, we expand $W$ as
\begin{equation}\label{the_sum}
W=W(x,y,z)=\sum_{\nu \geq 0} C_\nu\, \frac{(x-x^2y+x^2yz)^\nu}{(1+x-xy)^{2\nu+1}},
\end{equation}
whose terms are rational functions in $x$, $y$, and $z$ and there are no square roots involved.
We now expand each term of (\ref{the_sum}) using a multinomial or negative binomial expansion.
First, the multinomial expansion of the numerator takes the form 
\begin{equation}\label{expansion1}
(x-x^2y+x^2yz)^\nu = \sum_{i+j+m=\nu}\frac{\nu!}{i!\, j!\, m!}x^i(-x^2y)^j(x^2yz)^m=\sum_{i+j+m=\nu}\frac{(-1)^j\nu!}{i!\, j!\, m!}x^{i+2j+2m}y^{j+m}z^m.
\end{equation}
Next, the negative binomial expansion applied to the denominator gives
\begin{equation}\label{expansion2}
(1+x-yx)^{-(2\nu+1)}=\sum_s \frac{(2\nu+s)!}{(2\nu)!\, s!} (-x+yx)^s= \sum_{\ell,r}\frac{(2\nu+\ell+r)!}{(2\nu)!\, \ell !\,  r!}(-1)^\ell y^rx^{\ell+r}.
\end{equation}
Then, multiplying (\ref{expansion1}) and (\ref{expansion2}), summing with respect to $\nu$, and taking into account that by (\ref{expansion1}) we have $\nu=i+j+m$, Equation (\ref{the_sum}) becomes 
\begin{equation}\label{long_expansion}
W(x,y,z)=\sum_{i,j,m,\ell,r} \frac{(2i+2j+2m+\ell+r)!}{(i+j+m+1)!\, i!\, j!\, m!\, \ell!\, r!}(-1)^{j+\ell}x^{i+2j+2m+\ell+r}y^{j+m+r}z^m.
\end{equation}
Collecting terms in (\ref{long_expansion}), we get
\begin{align}
w_{n,k,m} &= \!\!\!\sum_{\substack{i+2j+2m+\ell+r=n\\ j+m+r=k} } \frac{(2i+2j+2m+\ell+r)!}{(i+j+m+1)!\, i!\, j!\, m!\, \ell!\, r!}(-1)^{j+\ell} \nonumber \\
&=\frac{1}{m!}\!\sum_{\substack{i+2j+2m+\ell+r=n\\j+m+r=k}} \frac{(-1)^{j+\ell}(2i+2j+2m+\ell+r)!}{(i+j+m+1)!\, i!\, j!\, \ell!\, r!}.\label{w2}
\end{align} 
Now, solving the system $\{i+2j+2m+\ell+r=n, j+m+r=k\}$ for $\ell, r\geq0$, the indexes $\ell$ and $r$ must satisfy the conditions
\begin{equation}\label{cond_ell_r}
\ell=n-k-m-i-j \geq 0 \quad \text{and} \quad r=k-m-j \geq 0,
\end{equation}
while $i,j\geq 0$.
This implies that $n$, $k$, and $m$ must satisfy the inequalities
\begin{equation}\label{ineq}
m\leq k \quad \text{and} \quad k+m \leq n 
\end{equation}
(otherwise the sum is empty and has value $0$). 
Substituting in (\ref{w2}) the values of $\ell$ and $r$ given by (\ref{cond_ell_r}), the coefficient $w_{n,k,m}$ is expressed as the double sum
\begin{equation}\label{double_sum}
w_{n,k,m} = \frac{1}{m!}\sum_{\substack{i+j\leq n-k-m \\j \leq k-m}} \frac{(-1)^{n-k-m-i}(n+i)!}{(i+j+m+1)!\, i!\, j!\, (n-k-m-i-j)!\, (k-m-j)!}. 
\end{equation}
This double sum can be greatly simplified using the Maple procedure ``\texttt{sum}" as follows. 
Since the factorials in the denominators in (\ref{double_sum}) are $\pm \infty$ when $i,j$ are big enough to be out of range, the corresponding terms in the double sum vanish.
Hence, we can take the ranges {\texttt{i=0..infty}} and {\texttt{j=0..infty}} in the double sum. 
Typing the Maple command
\vspace{0.2cm}

\noindent \texttt{> w[n,k,m] = 1/m!*sum(((-1)**(n-k-m-i)*(i+n)!/i!)*\\
sum(1/(i+j+m+1)!/j!/(n-i-j-k-m)!/(k-m-j)!,j=0..infinity),i=0..infinity);}
\vspace{0.2cm}\\
\noindent produces the output
\begin{equation*}\label{output}
w_{n,k,m} = -{\frac { \left( -1 \right) ^{n-k-m}n!\, \left( m+1 \right) m\,\sin
 \left( \pi\,k-\pi\,n \right) }{m!\, \left( m+1 \right) !\, \left( n-k
-m \right) !\, \left( k-m \right) !\,\sin \left( \pi\,m \right) 
 \left( k-n \right)  \left( k-1-n \right) }},
 \end{equation*}
which we rewrite in the form
\begin{equation}\label{wmnp+}
w_{n,k,m} = {\frac { \left( -1 \right) ^{n-k-m}n!\, \left( m+1 \right)  }{m!\, \left( m+1 \right) !\, \left( n-k
-m \right) !\, \left( k-m \right) !\, \left( n-k+1 \right) }}\cdot q(n,k,m)
\end{equation}
where $q(n,k,m)$ takes the indeterminate form
 \begin{equation*}
q(n,k,m)=\frac{m\sin(\pi (n-k))}{\sin(\pi m)(n-k)}=\frac{0}{0},
\end{equation*}
\noindent since $n$, $k$, and $m$ are positive integers. We ``lift" this indetermination by making use of the limits
\begin{align}
\lim_{t \rightarrow 0}\frac{t}{\sin(\pi t)}&=\frac{1}{\pi} \label{lim_i} \quad \text{and}\\ 
\lim_{t \rightarrow \pi}\frac{\sin(kt)}{\sin(\ell t)}&=(-1)^{k-\ell}\,\frac{k}{\ell}, \quad k \in \mathbb{Z}, \,\, 0\neq \ell \in \mathbb{Z}.\label{lim_ii}
\end{align}
To complete the proof, we break into cases: \vspace{5bp}
\begin{itemize}
    \item If $m=0$ and $k=n$, then using (\ref{lim_i}) twice, we get $q(k,k,0)=\frac{1}{\pi}\cdot \frac{\pi}{1}=1$. This implies, by (\ref{wmnp+}), that $w_{n,k,0}=1$ when $k=n$. \vspace{5bp}
    \item If $m=0$ and $k \neq n$, then by (\ref{ineq}) we need only to consider the case $k<n$. Hence, by (\ref{lim_i}) we get \[q(n,k,0)=\frac{1}{\pi}\cdot \frac{\sin(\pi(n-k))}{n-k}=0,\] since $\sin(\pi(n-k))=0$ and $n-k\neq 0$. This implies that $w_{k,0,n} = 0$ if $k \neq n$. \vspace{5bp}
    \item If $m>0$ and $k+m\leq n$, then we must have $k < n$. Thus, by (\ref{lim_ii}) with $k=n-k$ and $\ell=m$, we have
\begin{equation*}
q(n,k,m)=\frac{m}{n-k}\cdot \frac{\sin(\pi(n-k))}{\sin(\pi m)}=\frac{m}{n-k}\cdot (-1)^{n-k-m}.
\end{equation*}
Substituting this value into (\ref{wmnp+}) yields
\begin{equation*}
w_{n,k,m}={\frac {n!}{m!\, \left( m-1 \right) !\, \left( n-k-m \right) !\,
 \left( k-m \right) !\, \left( n-k \right)  \left( n-k+1 \right) }},
\end{equation*}
which is equivalent to the desired expression in Theorem ~\ref{t-explicit}. \qedhere
\end{itemize}
\end{proof}

In order to ``cross check" our formula in Theorem \ref{t-explicit}, we computed the first terms of $W(x,y,z)$ using two methods: 
firstly, by making use of the explicit expression for the coefficients given in Theorem \ref{t-explicit}, and 
secondly, by making use of the Maple ``\texttt{mtaylor}" command applied to (\ref{quad_sol}). 
Both methods gave the same results displayed in Table \ref{Tab_w}. 


\begin{table}[ht] 
\caption{The nonzero values of $w_{n,k,m}$ for $0\leq n\leq 10$.}\label{Tab_w}
\small
\begin{tabular}[t]{|c|c|}
\hline
$n,k,m$ & $w_{n,k,m}$\\ \hline \hline
0, 0, 0 & 1 \\ \hline
1, 1, 0 & 1 \\ \hline
2, 1, 1 & 1 \\ 
2, 2, 0 & 1 \\ \hline
3, 1, 1 & 1 \\
3, 2, 1 & 3 \\
3, 3, 0 & 1 \\ \hline
4, 1, 1 & 1 \\
4, 2, 1 & 4 \\
4, 2, 2 & 2 \\
4, 3, 1 & 6 \\
4, 4, 0 & 1 \\\hline
5, 1, 1 & 1 \\
5, 2, 1 & 5 \\
5, 2, 2 & 5 \\
5, 3, 1 & 10 \\
5, 3, 2 & 10 \\
5, 4, 1 & 10 \\
5, 5, 0 & 1 \\ \hline
6, 1, 1 & 1 \\
6, 2, 1 & 6 \\
6, 2, 2 & 9 \\
6, 3, 1 & 15 \\
6, 3, 2 & 30 \\
6, 3, 3 & 5 \\
6, 4, 1 & 20 \\
6, 4, 2 & 30 \\
6, 5, 1 & 15 \\
6, 6, 0 & 1 \\ \hline
\end{tabular}
\quad
\begin{tabular}[t]{|c|c|}
\hline
$n,k,m$ & $w_{n,k,m}$\\ \hline \hline
7, 1, 1 & 1 \\
7, 2, 1 & 7 \\
7, 2, 2 & 14 \\
7, 3, 1 & 21 \\
7, 3, 2 & 63 \\
7, 3, 3 & 21 \\
7, 4, 1 & 35 \\
7, 4, 2 & 105 \\
7, 4, 3 & 35 \\
7, 5, 1 & 35 \\
7, 5, 2 & 70 \\
7, 6, 1 & 21 \\
7, 7, 0 & 1 \\ \hline
8, 1, 1 & 1 \\
8, 2, 1 & 8 \\
8, 2, 2 & 20 \\
8, 3, 1 & 28 \\
8, 3, 2 & 112 \\
8, 3, 3 & 56 \\
8, 4, 1 & 56 \\
8, 4, 2 & 252 \\
8, 4, 3 & 168 \\
8, 4, 4 & 14 \\
8, 5, 1 & 70 \\
8, 5, 2 & 280 \\
8, 5, 3 & 140 \\
8, 6, 1 & 56 \\
8, 6, 2 & 140 \\
8, 7, 1 & 28 \\
8, 8, 0 & 1 \\ \hline
\end{tabular}
\quad 
\begin{tabular}[t]{|c|c|}
\hline
$n,k,m$ & $w_{n,k,m}$\\ \hline \hline
9, 1, 1 & 1 \\
9, 2, 1 & 9 \\
9, 2, 2 & 27 \\
9, 3, 1 & 36 \\
9, 3, 2 & 180 \\
9, 3, 3 & 120 \\
9, 4, 1 & 84 \\
9, 4, 2 & 504 \\
9, 4, 3 & 504 \\
9, 4, 4 & 84 \\
9, 5, 1 & 126 \\
9, 5, 2 & 756 \\
9, 5, 3 & 756 \\
9, 5, 4 & 126 \\
9, 6, 1 & 126 \\
9, 6, 2 & 630 \\
9, 6, 3 & 420 \\
9, 7, 1 & 84 \\
9, 7, 2 & 252 \\
9, 8, 1 & 36 \\
9, 9, 0 & 1 \\ \hline
\end{tabular} 
\quad
\begin{tabular}[t]{|c|c|}
\hline
$n,k,m$ & $w_{n,k,m}$\\ \hline \hline
10, 1, 1 & 1 \\
10, 2, 1 & 10 \\
10, 2, 2 & 35 \\
10, 3, 1 & 45 \\
10, 3, 2 & 270 \\
10, 3, 3 & 225 \\
10, 4, 1 & 120 \\
10, 4, 2 & 900 \\
10, 4, 3 & 1200 \\
10, 4, 4 & 300 \\
10, 5, 1 & 210 \\
10, 5, 2 & 1680 \\
10, 5, 3 & 2520 \\
10, 5, 4 & 840 \\
10, 5, 5 & 42 \\
10, 6, 1 & 252 \\
10, 6, 2 & 1890 \\
10, 6, 3 & 2520 \\
10, 6, 4 & 630 \\
10, 7, 1 & 210 \\
10, 7, 2 & 1260 \\
10, 7, 3 & 1050 \\
10, 8, 1 & 120 \\
10, 8, 2 & 420 \\
10, 9, 1 & 45 \\
10, 10, 0 & 1 \\ \hline
\end{tabular}\\ 
\end{table}


\end{document}